\documentclass[11pt]{article}

\usepackage[margin=1.05in]{geometry}
\usepackage{amsmath,amssymb,amsthm,mathtools}
\usepackage[T1]{fontenc}
\usepackage{lmodern}
\usepackage{mathrsfs}
\usepackage{bm}
\usepackage{enumitem}
\usepackage[colorlinks=true,linkcolor=blue,citecolor=blue,urlcolor=blue]{hyperref}

\numberwithin{equation}{section}

\newcommand{\R}{\mathbb{R}}
\newcommand{\N}{\mathbb{N}}

\newcommand{\Hs}{\mathcal{H}_{p,\mu}}
\newcommand{\Ls}{\mathcal{L}_{s}}

\newcommand{\supp}{\operatorname{supp}}
\newcommand{\loc}{\mathrm{loc}}
\newcommand{\PV}{\mathrm{P.V.}}

\newtheorem{theorem}{Theorem}[section]
\newtheorem{lemma}[theorem]{Lemma}
\newtheorem{proposition}[theorem]{Proposition}

\theoremstyle{definition}

\theoremstyle{remark}
\newtheorem{remark}[theorem]{Remark}

\title{Isolated Singularities for Fractional Hartree Equations}

\author{Guangze Gu, Aleks Jevnikar, Zhipeng Yang\thanks{Corresponding author: yangzhipeng326@163.com}}

\date{}

\hypersetup{
	pdftitle={Isolated Singularities for Fractional Hartree Equations},
	pdfauthor={Guangze Gu, Aleks Jevnikar, Zhipeng Yang},
	pdfsubject={Isolated singularities for fractional Hartree equations},
	pdfkeywords={Fractional Hartree equation, isolated singularity, moving spheres}
}

\AtEndDocument{%
	\par
	\bigskip
	\bigskip
	
	\noindent
	\textbf{Guangze Gu}\\[0.2em]
	\textsc{Department of Mathematics, Yunnan Normal University, Kunming 650500, China}\\
	\textsc{Yunnan Key Laboratory of Modern Analytical Mathematics and Applications, Kunming 650500, China}\\[0.3em]
	\textit{E-mail address}: \texttt{guangzegu@163.com}
	
	\bigskip
	
	\noindent
	\textbf{Aleks Jevnikar}\\[0.2em]
	\textsc{Department of Mathematics, Computer Science and Physics, University of Udine, Udine 33100, Italy}\\[0.3em]
	\textit{E-mail address}: \texttt{aleks.jevnikar@uniud.it}
	
	\bigskip
	
	\noindent
	\textbf{Zhipeng Yang}\\[0.2em]
	\textsc{Department of Mathematics, Yunnan Normal University, Kunming 650500, China}\\
	\textsc{Yunnan Key Laboratory of Modern Analytical Mathematics and Applications, Kunming 650500, China}\\[0.3em]
	\textit{E-mail address}: \texttt{yangzhipeng326@163.com}
}

\begin{document}
	
	\maketitle
	
	\begin{abstract}
		We study isolated singularities of positive solutions to a fractional Hartree equation with Riesz interaction,
		\[
		(-\Delta)^s u
		=
		\left(
		\int_{\mathbb{R}^N\setminus\{0\}}
		\frac{u^p(y)}{|x-y|^\mu}\,dy
		\right)u^q
		\quad \text{in } \mathbb{R}^N\setminus\{0\}.
		\]
		The puncture changes the passage from the differential equation to its integral form: a fractional fundamental solution term may occur at the singular point. Under an upper bound at the scale of the fundamental solution and a weighted integrability condition on the source, we derive a Riesz decomposition with a nonnegative singular term. Under the same assumptions, if in addition the solution blows up at the puncture, \(p,q\ge1\), and both Kelvin weights are nonnegative, an off-center moving spheres argument yields radial symmetry and strict radial decrease. We also classify positive radial homogeneous singular solutions under the corresponding convergence conditions. If a nonzero positive radial homogeneous scaling limit exists and solves the limiting equation, this classification uniquely determines its coefficient.
	\end{abstract}
	
	
	
	\section{Introduction and main results}
	\label{sec:intro}
	
	We study isolated singularities of positive solutions to the fractional Hartree equation
	\begin{equation}
		\label{eq1.1}
		(-\Delta)^s u
		=
		\left(
		\int_{\R^N\setminus\{0\}}
		\frac{u^p(y)}{|x-y|^\mu}\,dy
		\right)u^q
		\quad \text{in } \R^N\setminus\{0\},
	\end{equation}
	where
	\[
	0<s<1,\qquad N>2s,\qquad 0<\mu<N,\qquad p,q>0.
	\]
	Throughout the paper, we write
	\[
	\Hs[u](x)
	=
	\int_{\R^N\setminus\{0\}}
	\frac{u^p(y)}{|x-y|^\mu}\,dy
	\]
	for the Hartree potential generated by \(u\). Equation~\eqref{eq1.1} then takes the compact form
	\[
	(-\Delta)^s u=\Hs[u]u^q
	\quad \text{in } \R^N\setminus\{0\}.
	\]
	We use the term Hartree equation for the general interaction \(\Hs[u]u^q\), allowing \(p\) and \(q\) to vary independently. The standard Choquard equation is recovered, up to lower-order linear terms, in the variational subcase.
	
	To relate \eqref{eq1.1} to the usual Riesz potential notation, we use the convention
	\[
	a_{N,\eta}
	=
	\frac{\Gamma\left(\frac{N-\eta}{2}\right)}
	{2^\eta\pi^{N/2}\Gamma\left(\frac{\eta}{2}\right)},
	\qquad
	I_\eta(x)=a_{N,\eta}|x|^{\eta-N},
	\qquad 0<\eta<N.
	\]
	When the parameter denotes the denominator exponent, we write
	\[
	K_\mu(x)=|x|^{-\mu}
	=a_{N,N-\mu}^{-1}I_{N-\mu}(x).
	\]
	Thus the interaction in \eqref{eq1.1} is generated by \(K_\mu\). Hartree and Choquard equations form a standard class of nonlinear nonlocal elliptic problems. In the local setting, the classical variational Choquard equation
	\[
	-\Delta u+u
	=
	\left(I_\eta*|u|^p\right)|u|^{p-2}u
	\quad \text{in }\R^N
	\]
	goes back to the work of Lieb and Lions. Its qualitative theory includes existence, positivity, regularity, symmetry, and decay of ground states; see \cite{Lieb1977,Lions1980,MorozVanSchaftingen}. For fractional diffusion, d'Avenia, Siciliano and Squassina \cite{dAveniaSicilianoSquassina} established results on existence, regularity, symmetry, and decay. whole space Liouville and classification results for fractional Choquard and static Hartree equations were obtained in \cite{Le2019,DaiFangQin,DaiHuangQinWangFang,DaiLiuQin}.
	
	The presence of a puncture introduces a different set of questions. In the local semilinear theory, point-mass decompositions, removability, asymptotic symmetry, and the classification of singular profiles were developed in \cite{BrezisLions,BrezisVeron,CaffarelliGidasSpruck,VazquezVeron}. For fractional semilinear equations, Caffarelli, Jin, Sire and Xiong \cite{CaffarelliJinSireXiong} proved local asymptotic symmetry at an isolated singularity. Montoro, Punzo and Sciunzi \cite{MontoroPunzoSciunzi2018} treated qualitative properties in domains with zero-capacity singular sets, whereas Le \cite{Le2020} studied radial symmetry for a weighted singular Choquard equation driven by the fractional \(p\)-Laplacian under a different operator, weight, and asymptotic framework.
	
	B\^ocher and Liouville mechanisms for fractional equations provide the natural background for identifying a point defect. Li, Wu and Xu \cite{LiWuXu2018} and Li, Liu, Wu and Xu \cite{LiLiuWuXu2020} proved B\^ocher-type results for nonnegative fractional superharmonic or fractional harmonic solutions near an isolated point. The global theorem of Chen, D'Ambrosio and Li \cite[Theorem~1.3]{ChenDAmbrosioLi2015} identifies a distributional \(2s\)-harmonic function in the corresponding weighted \(L^1\) class as an affine function, with only a constant allowed when \(2s\le1\). Kim and Lee \cite{KimLee2025} obtained more recent B\^ocher and Liouville theorems under explicit one-sided assumptions. Since these theorems do not directly cover the sign-changing remainder used here, Section~\ref{sec:prelim} supplies the required arguments for point-supported distributions and their Fourier transforms.
	
	Terms involving the fundamental solution also occur in the theory of local Choquard and Hartree equations. Chen and Zhou \cite{ChenZhou2016,ChenZhou2023} classified isolated singularities and studied qualitative properties of singular solutions. Ghergu and Taliaferro \cite{GherguTaliaferro2016} obtained sharp pointwise bounds for Choquard--Pekar inequalities, while Wang \cite{Wang2017} treated the defocusing Hartree equation through a distributional point-mass formulation. Cai, Gu and Jevnikar \cite{CaiGuJevnikar2025} considered a local Hartree equation with a compact zero-capacity singular set and proved symmetry and monotonicity by moving-plane methods. Ghergu and Yu \cite{GherguYu2026} recently studied convolution equations in a punctured ball and obtained sharp bounds at the scale of the fundamental solution in appropriate parameter ranges.
	
	For the critical local Hartree equation, Andrade, Feng, Piccione and Yang \cite{AndradeFengPiccioneYang2025} proved radiality and monotonicity in the entire punctured space and sharp local asymptotics in a punctured ball, the latter by blow-up analysis and an asymptotic integral moving spheres argument. Feng, Yang and Zhou \cite{FengYangZhou2026} obtained related asymptotic results for a planar exponential Hartree equation. Boni, Noja and Scandone \cite{BoniNojaScandone2026} related isolated singular solutions of local semilinear equations in dimensions two and three to nonlinear Schr\"odinger operators with point interactions.
	
	Among whole space fractional results, the work of Dai, Liu and Qin \cite{DaiLiuQin} is closest to the present algebraic setting. They treated the same structure with independent exponents and the same two Kelvin weights, established the correspondence between the differential and integral formulations, and classified regular entire solutions starting from an integral representation. Here the domain is punctured, the representation must be derived from the differential equation, and the resulting fundamental-solution component must be retained throughout an off-center Kelvin comparison.
	
	After Kelvin reflection, the Hartree interaction produces a nonlocal defect involving a second convolution. Its positive part is generated only by the set on which the Kelvin comparison fails. This localization yields a small-set estimate and a narrow-region principle compatible with the nonnegative fundamental-solution term.
	
	We prove three results. First, the punctured equation determines a Riesz decomposition with a nonnegative point mass under a local growth bound and a weighted condition on the source. Second, this representation can be incorporated into an off-center moving spheres argument to obtain radial symmetry and strict radial decrease. Third, we classify all solutions in the positive radial homogeneous class under the corresponding convergence conditions.
	
	Before stating the main results, we record the scaling and Kelvin parameters associated with \eqref{eq1.1}. When \(p+q\ne1\), its natural scaling is
	\[
	u_\rho(x)=\rho^\alpha u(\rho x),
	\]
	where
	\begin{equation}
		\label{eq1.2}
		\alpha
		=
		\frac{N-\mu+2s}{p+q-1}.
	\end{equation}
	The balance corresponding to the fundamental solution, \(\alpha=N-2s\), is equivalent to
	\[
	p+q=\frac{2N-\mu}{N-2s}.
	\]
	This balance should not be confused with exact Kelvin covariance, which requires the simultaneous vanishing of
	\[
	\sigma_p=2N-\mu-(N-2s)p,
	\qquad
	\sigma_q=N+2s-\mu-(N-2s)q.
	\]
	It therefore occurs only at
	\[
	p=\frac{2N-\mu}{N-2s},
	\qquad
	q=\frac{N+2s-\mu}{N-2s}.
	\]
	At this pair, the scaling exponent in \eqref{eq1.2} is \(\alpha=(N-2s)/2\).
	
	For the moving spheres theorem, we work in the Kelvin monotone range
	\begin{equation}
		\label{eq1.3}
		p,q\ge1,
		\qquad
		p\le \frac{2N-\mu}{N-2s},
		\qquad
		q\le \frac{N+2s-\mu}{N-2s}.
	\end{equation}
	Equivalently, \eqref{eq1.3} consists of \(p,q\ge1\) together with
	\[
	\sigma_p\ge0,
	\qquad
	\sigma_q\ge0.
	\]
	Thus both Kelvin weights are nonnegative. Since \(q\ge1\), the range \eqref{eq1.3} is nonempty only if \(\mu\le4s\).
	
	We next fix the functional setting and normalization used in the decomposition theorem. Set
	\[
	\Ls(\R^N)
	=
	\left\{
	u\in L^1_{\loc}(\R^N):
	\int_{\R^N}\frac{|u(x)|}{1+|x|^{N+2s}}\,dx<+\infty
	\right\}.
	\]
	We use the Fourier convention
	\[
	\widehat f(\xi)=\int_{\R^N}e^{-ix\cdot\xi}f(x)\,dx,
	\qquad
	\widehat{(-\Delta)^s\varphi}(\xi)
	=|\xi|^{2s}\widehat{\varphi}(\xi),
	\]
	and the singular integral normalization
	\[
	C_{N,s}
	=
	\frac{4^s\Gamma\left(\frac N2+s\right)}
	{\pi^{N/2}|\Gamma(-s)|}.
	\]
	For the fundamental solution, set
	\[
	c_{N,s}
	=
	\frac{\Gamma\left(\frac{N-2s}{2}\right)}
	{2^{2s}\pi^{N/2}\Gamma(s)},
	\qquad
	\Phi_s(x)=c_{N,s}|x|^{-(N-2s)}.
	\]
	Then
	\[
	(-\Delta)^s\Phi_s=\delta_0
	\quad\text{in }\mathcal D'(\R^N).
	\]
	We call \(u\) a positive punctured solution of \eqref{eq1.1} if
	\[
	u\in \Ls(\R^N)\cap C^{2s+\beta}_{\loc}(\R^N\setminus\{0\})
	\]
	for some \(\beta>0\), with \(2s+\beta\notin\N\), \(u>0\) in \(\R^N\setminus\{0\}\), \(\Hs[u]\) is finite at every point of \(\R^N\setminus\{0\}\), and \eqref{eq1.1} holds pointwise there. The condition \(u\in\Ls(\R^N)\) is included so that the fractional Laplacian is well defined in the standard whole space framework; see \cite{Silvestre2006}.
	
	Our first theorem identifies the distributional defect created by the puncture.
	
	\begin{theorem}
		\label{Thm1.1}
		Let
		\[
		0<s<1,\qquad N>2s,\qquad 0<\mu<N,\qquad p,q>0.
		\]
		Let \(u\) be a positive punctured solution of \eqref{eq1.1}, and set
		\[
		F=\Hs[u]u^q.
		\]
		After assigning any value to \(F\) at the origin, assume that the Hartree source satisfies
		\begin{equation}
			\label{eq:source-condition}
			\int_{\R^N}
			\frac{F(y)}{1+|y|^{N-2s}}\,dy<+\infty.
		\end{equation}
		Assume also that the singularity is at most of fundamental-solution order:
		\begin{equation}
			\label{eq:fundamental-bound}
			\limsup_{x\to0}|x|^{N-2s}u(x)<+\infty.
		\end{equation}
		Define the Riesz potential of the Hartree source by
		\begin{equation}
			\label{eq:Riesz potential}
			V(x)=c_{N,s}\int_{\R^N}\frac{F(y)}{|x-y|^{N-2s}}\,dy,
			\qquad x\in\R^N\setminus\{0\}.
		\end{equation}
		Then there exists \(m\ge0\) such that
		\begin{equation}
			\label{eq:decomposition}
			u(x)=V(x)+m\Phi_s(x),
			\qquad x\in\R^N\setminus\{0\}.
		\end{equation}
		Consequently,
		\[
		(-\Delta)^s u=F+m\delta_0
		\quad\text{in }\mathcal D'(\R^N).
		\]
	\end{theorem}
	
	The coefficient \(m\) in Theorem~\ref{Thm1.1} is the mass of the distributional defect; the coefficient of \(|x|^{-(N-2s)}\) is \(mc_{N,s}\). The representation \eqref{eq:decomposition} is derived from the punctured equation and is the starting point for the symmetry argument.
	
	For \(x\ne0\) and \(0<\lambda<|x|\), consider the Kelvin transform
	\[
	u_{x,\lambda}(y)
	=
	\left(\frac{\lambda}{|y-x|}\right)^{N-2s}
	u\left(x+\frac{\lambda^2(y-x)}{|y-x|^2}\right).
	\]
	Under this transformation, the Riesz potential part of \eqref{eq:decomposition} produces the weights \(\sigma_p\) and \(\sigma_q\), while the term \(m\Phi_s\) contributes a nonnegative remainder. To ensure that the limiting sphere reaches the puncture, we additionally impose the blow-up condition
	\begin{equation}
		\label{eq:blowup}
		\lim_{x\to0}u(x)=+\infty.
	\end{equation}
	
	\begin{theorem}
		\label{Thm1.2}
		Assume the hypotheses of Theorem~\ref{Thm1.1}, with \(\mu\le4s\). Suppose in addition that
		\eqref{eq1.3} and \eqref{eq:blowup} hold. Then \(u\) is radially symmetric
		about the origin. More precisely, there exists a positive function
		\(U:(0,+\infty)\to(0,+\infty)\) such that
		\[
		u(x)=U(|x|)
		\quad \text{for every }x\in\R^N\setminus\{0\},
		\]
		and \(U\) is strictly decreasing.
	\end{theorem}
	
	The preceding symmetry theorem relies on \eqref{eq1.3}, under which both Kelvin weights are nonnegative. The explicit homogeneous calculation below is independent of that restriction and uses only the convergence conditions stated in the theorem.
	
	\begin{theorem}
		\label{Thm1.3}
		Let
		\[
		0<s<1,
		\qquad
		N>2s,
		\qquad
		0<\mu<N,
		\qquad
		p,q>0,
		\qquad
		p+q>1.
		\]
		Set
		\begin{equation}
			\label{eq:hom-alpha}
			\alpha
			=
			\frac{N-\mu+2s}{p+q-1}
		\end{equation}
		and assume that
		\begin{equation}
			\label{eq:hom-convergence}
			0<\alpha<N-2s,
			\qquad
			0<\alpha p<N,
			\qquad
			\mu+\alpha p>N.
		\end{equation}
		Define
		\begin{equation}
			\label{eq:lambda-coefficient}
			\Lambda_{N,s}(\alpha)
			=
			2^{2s}
			\frac{
				\Gamma\left(\frac{N-\alpha}{2}\right)
				\Gamma\left(\frac{\alpha+2s}{2}\right)
			}{
				\Gamma\left(\frac{\alpha}{2}\right)
				\Gamma\left(\frac{N-\alpha-2s}{2}\right)
			}
			>0
		\end{equation}
		and
		\begin{equation}
			\label{eq:hartree-coefficient}
			\mathcal C_{N,\mu,\alpha,p}
			=
			\pi^{\frac{N}{2}}
			\frac{
				\Gamma\left(\frac{N-\mu}{2}\right)
				\Gamma\left(\frac{N-\alpha p}{2}\right)
				\Gamma\left(\frac{\mu+\alpha p-N}{2}\right)
			}{
				\Gamma\left(\frac{\mu}{2}\right)
				\Gamma\left(\frac{\alpha p}{2}\right)
				\Gamma\left(\frac{2N-\mu-\alpha p}{2}\right)
			}
			>0.
		\end{equation}
		Then there is a unique \(A>0\) for which
		\[
		u_*(x)=A|x|^{-\alpha}
		\]
		solves \eqref{eq1.1} in \(\R^N\setminus\{0\}\). It is determined by
		\begin{equation}
			\label{eq:hom-coefficient}
			A^{p+q-1}
			=
			\frac{\Lambda_{N,s}(\alpha)}
			{\mathcal C_{N,\mu,\alpha,p}}.
		\end{equation}
		Conversely, suppose that \(u(x)=B|x|^{-\beta}\), with \(B,\beta>0\), is a positive radial homogeneous solution and that
		\[
		0<\beta<N-2s,\qquad
		0<\beta p<N,\qquad
		\mu+\beta p>N.
		\]
		Then \(\beta=\alpha\) and \(B=A\).
	\end{theorem}

	The assumptions in Theorems~\ref{Thm1.1} and~\ref{Thm1.2} have different roles.
	Condition~\eqref{eq:source-condition} controls the source near the puncture and at infinity,
	whereas \eqref{eq:fundamental-bound} is a local growth bound. For general punctured
	solutions, both are retained as explicit hypotheses. The blow-up
	condition~\eqref{eq:blowup} follows from \eqref{eq:decomposition} whenever \(m>0\),
	since \(u\ge m\Phi_s\). When \(m=0\), it is retained as a separate hypothesis in
	Theorem~\ref{Thm1.2}.
	
	All three conditions hold for the homogeneous solutions of Theorem~\ref{Thm1.3}.
	Indeed, for \(u_*(x)=A|x|^{-\alpha}\), the source is
	\[
	F_*(x)=\Hs[u_*](x)u_*^q(x)
	=A\Lambda_{N,s}(\alpha)|x|^{-\alpha-2s}.
	\]
	Since \(0<\alpha<N-2s\),
	\[
	\int_{\R^N}\frac{F_*(y)}{1+|y|^{N-2s}}\,dy
	\le C\left(
	\int_0^1 r^{N-1-\alpha-2s}\,dr
	+\int_1^\infty r^{-\alpha-1}\,dr
	\right)<+\infty,
	\]
	and
	\[
	\lim_{x\to0}|x|^{N-2s}u_*(x)=0,
	\qquad
	\lim_{x\to0}u_*(x)=+\infty.
	\]
	In particular, for any \(0<s<1\), \(N>2s\), \(0<\mu<N\), and \(\mu\le4s\), take
	\[
	p=\frac{2N-\mu}{N-2s},
	\qquad
	q=\frac{N+2s-\mu}{N-2s},
	\qquad
	\alpha=\frac{N-2s}{2}.
	\]
	Then \(p,q\ge1\), both Kelvin weights vanish, and
	\[
	\alpha p=N-\frac{\mu}{2}<N,
	\qquad
	\mu+\alpha p=N+\frac{\mu}{2}>N.
	\]
	Thus \(u_*(x)=A|x|^{-(N-2s)/2}\), with \(A\) given by
	\eqref{eq:hom-coefficient}, satisfies every hypothesis of Theorem~\ref{Thm1.2}.
	The limit above and \(u_*\ge m\Phi_s\) also give \(m=0\), so blow-up does not
	require a positive point mass.
	
	The rest of the paper is organized as follows. Section~\ref{sec:prelim} proves Theorem~\ref{Thm1.1} through weighted Riesz estimates and a point-supported distribution argument. Section~\ref{sec:moving} uses this decomposition in an off-center moving spheres proof of Theorem~\ref{Thm1.2}. Section~\ref{sec:profiles} proves Theorem~\ref{Thm1.3}. 
	
	\section{Preliminaries and Riesz decomposition}
	\label{sec:prelim}
	
	This section proves Theorem~\ref{Thm1.1}. We first establish the weighted
	Riesz estimates and the point-supported distribution lemma needed to
	identify the singular part.
	
	For \(u\in\Ls(\R^N)\) sufficiently regular near \(x\), the fractional Laplacian is defined by
	\[
	(-\Delta)^s u(x)
	=
	C_{N,s}\PV
	\int_{\R^N}
	\frac{u(x)-u(y)}{|x-y|^{N+2s}}\,dy.
	\]
	For nonsmooth functions, the operator is understood in the distributional sense:
	\[
	\langle (-\Delta)^s u,\varphi\rangle
	=
	\int_{\R^N}u(x)(-\Delta)^s\varphi(x)\,dx,
	\qquad
	\varphi\in C_c^\infty(\R^N).
	\]
	This pairing is well defined for \(u\in\Ls(\R^N)\), since
	\[
	|(-\Delta)^s\varphi(x)|
	\le
	\frac{C_\varphi}{1+|x|^{N+2s}}
	\quad \text{for }x\in\R^N .
	\]
	
	\begin{lemma}
		\label{Lem2.1}
		Set \(\gamma=N-2s\). There exists \(C>0\), depending only on \(N\) and \(s\), such that
		\begin{equation}
			\label{eq2.1}
			\int_{\R^N}
			\frac{dz}{|z-y|^\gamma(1+|z|^{N+2s})}
			\le
			\frac{C}{1+|y|^\gamma}
			\qquad\text{for every }y\in\R^N.
		\end{equation}
		Moreover, for every \(\varphi\in C_c^\infty(\R^N)\), there exists
		\(C_\varphi>0\) such that
		\begin{equation}
			\label{eq2.2}
			\int_{\R^N}
			\frac{|(-\Delta)^s\varphi(z)|}{|z-y|^\gamma}\,dz
			\le
			\frac{C_\varphi}{1+|y|^\gamma}
			\qquad\text{for every }y\in\R^N.
		\end{equation}
	\end{lemma}
	
	\begin{proof}
		We first prove \eqref{eq2.1}. If \(|y|\le2\), split the integral into
		\(\{|z|\le4\}\) and \(\{|z|>4\}\). The singularity on the first set is
		uniformly integrable because \(0<\gamma<N\). On the second set,
		\(|z-y|\ge |z|/2\), and hence the integrand is bounded by
		\(C|z|^{-2N}\). Thus the integral is uniformly bounded for \(|y|\le2\).
		
		Let \(R=|y|>2\), and decompose \(\R^N\) into
		\[
		E_1=\{|z|\le R/2\},\qquad
		E_2=\{|z-y|\le R/2\},\qquad
		E_3=\R^N\setminus(E_1\cup E_2).
		\]
		On \(E_1\), one has \(|z-y|\ge R/2\), so
		\[
		\int_{E_1}
		\frac{dz}{|z-y|^\gamma(1+|z|^{N+2s})}
		\le CR^{-\gamma}.
		\]
		On \(E_2\), one has \(|z|\ge R/2\), and therefore
		\[
		\int_{E_2}
		\frac{dz}{|z-y|^\gamma(1+|z|^{N+2s})}
		\le
		CR^{-N-2s}\int_{|w|\le R/2}|w|^{-\gamma}\,dw
		\le CR^{-N}.
		\]
		Finally, \(|z-y|\ge R/2\) and \(|z|\ge R/2\) on \(E_3\), whence
		\[
		\int_{E_3}
		\frac{dz}{|z-y|^\gamma(1+|z|^{N+2s})}
		\le
		CR^{-\gamma}
		\int_{|z|\ge R/2}\frac{dz}{1+|z|^{N+2s}}
		\le CR^{-N}.
		\]
		Since \(N>\gamma\), these estimates prove \eqref{eq2.1}.
		
		We next recall the decay of the fractional Laplacian of a test function:
		\begin{equation}
			\label{eq2.3}
			|(-\Delta)^s\varphi(z)|
			\le C_\varphi(1+|z|)^{-N-2s}.
		\end{equation}
		Indeed, \((-\Delta)^s\varphi\) is bounded on every fixed ball. If
		\(\supp\varphi\subset B_R\) and \(|z|>2R\), then \(\varphi(z)=0\) and no
		principal value is needed, so
		\[
		|(-\Delta)^s\varphi(z)|
		\le
		C_{N,s}\int_{B_R}
		\frac{|\varphi(w)|}{|z-w|^{N+2s}}\,dw
		\le C_\varphi |z|^{-N-2s}.
		\]
		Since \((1+|z|)^{N+2s}\) and \(1+|z|^{N+2s}\) are comparable,
		\eqref{eq2.2} follows from \eqref{eq2.1} and \eqref{eq2.3}.
	\end{proof}
	
	\begin{lemma}
		\label{Lem2.2}
		Let \(F\ge0\) be locally integrable and satisfy
		\[
		\int_{\R^N}\frac{F(y)}{1+|y|^{N-2s}}\,dy<+\infty.
		\]
		Define
		\[
		V(x)=\int_{\R^N}\Phi_s(x-y)F(y)\,dy
		=
		c_{N,s}\int_{\R^N}\frac{F(y)}{|x-y|^{N-2s}}\,dy.
		\]
		Then \(V\in\Ls(\R^N)\) and
		\begin{equation}
			\label{eq2.4}
			(-\Delta)^sV=F
			\quad\text{in }\mathcal D'(\R^N).
		\end{equation}
		If \(F\) is locally bounded in \(\R^N\setminus\{0\}\), then
		\[
		V\in C(\R^N\setminus\{0\})
		\quad\text{and}\quad
		V(x)<+\infty\quad\text{for every }x\ne0.
		\]
	\end{lemma}
	
	\begin{proof}
		The nonnegativity of \(F\) permits Tonelli's theorem. Lemma~\ref{Lem2.1}
		gives
		\begin{align*}
			\int_{\R^N}\frac{V(x)}{1+|x|^{N+2s}}\,dx
			&=
			c_{N,s}\int_{\R^N}F(y)
			\left(
			\int_{\R^N}
			\frac{dx}{|x-y|^{N-2s}(1+|x|^{N+2s})}
			\right)dy\\
			&\le
			C\int_{\R^N}\frac{F(y)}{1+|y|^{N-2s}}\,dy
			<+\infty.
		\end{align*}
		Thus \(V\in\Ls(\R^N)\).
		
		Let \(\varphi\in C_c^\infty(\R^N)\). Before exchanging the integrations,
		we use \eqref{eq2.2} to check that
		\begin{align*}
			&c_{N,s}\int_{\R^N}F(y)
			\left(
			\int_{\R^N}
			\frac{|(-\Delta)^s\varphi(x)|}{|x-y|^{N-2s}}\,dx
			\right)dy\\
			&\hspace{4em}\le
			C_\varphi\int_{\R^N}
			\frac{F(y)}{1+|y|^{N-2s}}\,dy
			<+\infty.
		\end{align*}
		Fubini's theorem is therefore applicable. Since
		\((-\Delta)^s\Phi_s=\delta_0\) with the Fourier normalization fixed in
		Section~\ref{sec:intro}, translation invariance gives
		\[
		\int_{\R^N}\Phi_s(x-y)(-\Delta)^s\varphi(x)\,dx
		=\varphi(y).
		\]
		Consequently,
		\[
		\int_{\R^N}V(x)(-\Delta)^s\varphi(x)\,dx
		=
		\int_{\R^N}F(y)\varphi(y)\,dy,
		\]
		which proves \eqref{eq2.4}.
		
		It remains to justify the pointwise assertion. Fix \(x\ne0\) and choose
		\(0<\rho<|x|/4\). On \(B_\rho(x)\), the source is bounded and the Riesz
		kernel is locally integrable. On bounded sets separated from \(x\), the
		kernel is bounded and \(F\in L^1_{\loc}\). For \(|y|>2|x|\), one has
		\(|x-y|\ge |y|/2\), and the remaining tail is controlled by the weighted
		source integral. Hence \(V(x)<+\infty\).
		
		We finally prove continuity away from the origin. Let \(x_j\to x\ne0\).
		Choose \(0<\delta<|x|/8\), and take \(j\) sufficiently large that
		\(|x_j-x|<\delta/2\). Local boundedness of \(F\), together with
		\[
		\int_{B_{3\delta}(0)}
		\frac{dz}{|z|^{N-2s}}
		\le C\delta^{2s},
		\]
		shows that the contributions near \(x\) and \(x_j\) are
		\(O(\delta^{2s})\), uniformly in \(j\). On
		\(B_R\setminus B_{2\delta}(x)\), dominated convergence applies because
		\(F\in L^1_{\loc}(\R^N)\) and, for large \(j\), both kernels are
		uniformly separated from their singularities. If \(R\) is sufficiently
		large, then
		\[
		\frac{1}{|x_j-y|^{N-2s}}+\frac{1}{|x-y|^{N-2s}}
		\le \frac{C}{|y|^{N-2s}}
		\quad\text{for }|y|>R,
		\]
		uniformly for large \(j\), and the weighted source condition makes this
		tail uniformly small. The resulting three-region estimate, followed by
		\(\delta\to0\) and \(R\to\infty\), gives \(V(x_j)\to V(x)\).
	\end{proof}
	
	\begin{lemma}
		\label{Lem2.3}
		Let \(h\in\Ls(\R^N)\) satisfy
		\[
		(-\Delta)^s h=0
		\quad \text{in }\R^N\setminus\{0\}
		\]
		in the distributional sense. Assume that
		\[
		h(x)\le C|x|^{-(N-2s)}
		\quad \text{for a.e. }0<|x|<r_0
		\]
		for some \(r_0>0\). Then
		\[
		h(x)=m\Phi_s(x)+P(x)
		\quad\text{for a.e. }x\in\R^N\setminus\{0\},
		\]
		where \(m\in\R\) and \(P\) is a polynomial of degree strictly less than
		\(2s\). In particular,
		\[
		P(x)=b\quad\text{if }0<s\le\frac12,
		\]
		whereas
		\[
		P(x)=a\cdot x+b\quad\text{if }\frac12<s<1.
		\]
	\end{lemma}
	
	\begin{proof}
		Set \(\gamma=N-2s\). Since \(h\in\Ls(\R^N)\), it defines a tempered
		distribution. The distribution
		\[
		T=(-\Delta)^s h
		\]
		is supported at the origin. We first make the reduction to a point-supported
		distribution explicit. A compactly supported distribution has finite order, say \(M\):
		for a fixed compact neighborhood \(K\) of the origin, its continuity gives
		\[
		|\langle T,\zeta\rangle|
		\le C_K\max_{|\alpha|\le M}\|D^\alpha\zeta\|_{L^\infty(K)}
		\quad\text{for }\zeta\in C_c^\infty(K).
		\]
		Choose a cutoff \(\chi\) equal to one near the origin. For a test function
		\(\psi\), let \(P_M\psi\) be its Taylor polynomial of order \(M\) at the
		origin and put \(R=\chi(\psi-P_M\psi)\). The function \(R\) and all its
		derivatives of order at most \(M\) vanish at the origin. Since
		\(\supp T\subset\{0\}\), one has
		\(\langle T,R\rangle=\langle T,\chi(\cdot/\varepsilon)R\rangle\) for every
		sufficiently small \(\varepsilon\). Taylor's formula and the finite-order
		bound give
		\[
		\bigl|\langle T,\chi(\cdot/\varepsilon)R\rangle\bigr|
		\le C\varepsilon ,
		\]
		which tends to zero. Hence \(\langle T,R\rangle=0\), while
		\(\langle T,\psi\rangle=\langle T,\chi\psi\rangle\) because \(\chi=1\)
		near \(\supp T\). It follows that
		\(\langle T,\psi\rangle=\langle T,\chi P_M\psi\rangle\), which depends only
		on the derivatives \(D^\alpha\psi(0)\), \(|\alpha|\le M\). Therefore
		\begin{equation}
			\label{eq2.5}
			T=\sum_{|\alpha|\le M}c_\alpha D^\alpha\delta_0
			\quad\text{in }\mathcal D'(\R^N).
		\end{equation}
		Both sides are tempered distributions.
		
		Regard each \(D^\alpha\Phi_s\) as the distributional derivative of
		\(\Phi_s\in\mathcal S'(\R^N)\), and set
		\[
		g=h-\sum_{|\alpha|\le M}c_\alpha D^\alpha\Phi_s
		\quad\text{in }\mathcal S'(\R^N).
		\]
		We identify the Fourier support of \(g\) directly, without multiplying an
		arbitrary tempered distribution by \(|\xi|^{2s}\). We first establish the
		Schwartz-space estimate used in this argument. There is a Schwartz seminorm
		\(p_L\) such that
		\begin{equation}
			\label{eq2.6}
			\sup_{x\in\R^N}
			(1+|x|^{N+2s})|(-\Delta)^s\vartheta(x)|
			\le C p_L(\vartheta)
			\qquad\text{for }\vartheta\in\mathcal S(\R^N).
		\end{equation}
		Indeed, on a fixed ball the principal-value integral is controlled by
		second derivatives and an integrable tail. For large \(|x|\), split the
		integral into \(|y|\le|x|/2\), \(|y-x|<|x|/2\), and the remaining region.
		The first part is bounded by
		\(C|x|^{-N-2s}\|\vartheta\|_{L^1}\) for the \(\vartheta(y)\) term, while
		\[
		|\vartheta(x)|
		\int_{|y|\le |x|/2}|x-y|^{-N-2s}\,dy
		\le C|x|^{-2s}|\vartheta(x)|
		\le Cp_N(\vartheta)|x|^{-N-2s}.
		\]
		Here \(p_N(\vartheta):=\sup_z(1+|z|)^N|\vartheta(z)|\).
		In the second region, the
		principal-value Taylor remainder is controlled by rapidly decreasing first
		and second derivatives, while in the third region both kernel separation and
		rapid decay apply. These estimates use only finitely many Schwartz
		seminorms and prove \eqref{eq2.6}.
		
		Let \(\chi_R(x)=\chi(x/R)\), where \(\chi=1\) near the origin. Since
		\(\chi_R\vartheta\to\vartheta\) in \(\mathcal S\), \eqref{eq2.6} and
		\(h\in\Ls(\R^N)\) give
		\[
		\int_{\R^N}h(-\Delta)^s(\chi_R\vartheta)\,dx
		\longrightarrow
		\int_{\R^N}h(-\Delta)^s\vartheta\,dx.
		\]
		On the other hand, the definition of \(T=(-\Delta)^sh\) gives
		\[
		\int_{\R^N}h(-\Delta)^s(\chi_R\vartheta)\,dx
		=\langle T,\chi_R\vartheta\rangle.
		\]
		Because \(\supp T\subset\{0\}\), the right-hand side equals
		\(\langle T,\vartheta\rangle\) for all sufficiently large \(R\). Therefore
		\begin{equation}
			\label{eq2.7}
			\int_{\R^N}h(-\Delta)^s\vartheta\,dx
			=\langle T,\vartheta\rangle
			\qquad\text{for every }\vartheta\in\mathcal S(\R^N).
		\end{equation}
		
		Let \(\psi\in C_c^\infty(\R^N\setminus\{0\})\) and choose
		\(\vartheta\in\mathcal S(\R^N)\) by
		\[
		\widehat{\vartheta}(\xi)
		=(2\pi)^N|\xi|^{-2s}\psi(-\xi).
		\]
		Its Fourier transform is compactly supported away from the origin. Hence
		\((-\Delta)^s\vartheta=\widehat\psi\), because
		\(\widehat{\widehat\psi}(\xi)=(2\pi)^N\psi(-\xi)\). On this frequency
		support, the identity
		\[
		\langle D^\alpha\Phi_s,(-\Delta)^s\vartheta\rangle
		=\langle D^\alpha\delta_0,\vartheta\rangle
		\]
		follows directly from \(\widehat{\Phi_s}=|\xi|^{-2s}\). Combining it with
		\eqref{eq2.5} and \eqref{eq2.7}, and using the distributional convention
		\(\langle\widehat g,\psi\rangle=\langle g,\widehat\psi\rangle\), gives
		\[
		\langle\widehat g,\psi\rangle
		=\langle g,(-\Delta)^s\vartheta\rangle=0.
		\]
		It follows that \(\supp\widehat g\subset\{0\}\). Applying the argument used
		to obtain \eqref{eq2.5} to \(\widehat g\), we conclude that \(g=P\)
		is a polynomial. Thus
		\begin{equation}
			\label{eq2.8}
			h=\sum_{|\alpha|\le M}c_\alpha D^\alpha\Phi_s+P
		\end{equation}
		as tempered distributions, and hence almost everywhere away from the origin.
		
		We now exclude every derivative of positive order. Let \(k\ge1\) be the
		largest order occurring in \eqref{eq2.8}. The sum of the order-\(k\) terms
		is a homogeneous function of degree \(-(\gamma+k)\) on the punctured
		space. It cannot vanish identically unless all its coefficients vanish:
		its Fourier transform away from the origin has the form
		\(Q_k(\xi)|\xi|^{-2s}\), where \(Q_k\) is the corresponding homogeneous
		polynomial. If \(k\ge2s\), then \(\gamma+k\ge N\). On a set of directions
		where the angular factor is bounded away from zero, the absolute radial
		integral is
		\[
		\int_0^{r_0}r^{N-1-\gamma-k}\,dr
		=
		\int_0^{r_0}r^{2s-k-1}\,dr
		=+\infty.
		\]
		Neither the lower-order homogeneous terms nor the polynomial can cancel
		this leading order. This contradicts \(h\in L^1_{\loc}(\R^N)\). Hence
		\(k<2s\).
		
		Since \(0<2s<2\), the only remaining positive order is \(k=1\), and this
		can occur only when \(s>1/2\). Such a term has the form
		\[
		a\cdot\nabla\Phi_s(x)
		=
		-c_{N,s}\gamma(a\cdot x)|x|^{-\gamma-2}.
		\]
		If \(a\ne0\), this expression is bounded below by
		\(c|x|^{-\gamma-1}\) on a cone. The remaining terms have lower singular
		order, so the assumed one-sided bound \(h(x)\le C|x|^{-\gamma}\) is
		violated on that cone. Therefore every derivative term of positive order
		vanishes, and
		\[
		h=m\Phi_s+P.
		\]
		
		Finally, \(\Phi_s\in\Ls(\R^N)\): its radial integral near the origin is a
		multiple of \(\int_0^1r^{2s-1}\,dr\), while its weighted tail is a multiple
		of \(\int_1^\infty r^{-N-1}\,dr\). Hence
		\(P=h-m\Phi_s\in\Ls(\R^N)\). If \(P\) has degree \(d\), its
		leading homogeneous part is bounded away from zero on some cone, and its
		tail contribution to the \(\Ls\)-norm dominates
		\[
		\int_1^\infty r^{d-2s-1}\,dr.
		\]
		Thus \(d<2s\). In the borderline case \(s=1/2\), an affine term gives the
		logarithmically divergent integral \(\int_1^\infty r^{-1}\,dr\).
		The stated alternatives for \(P\) follow.
	\end{proof}
	
	\begin{proof}[Proof of Theorem~\ref{Thm1.1}]
		We first justify the passage from the pointwise equation to the distributional
		one. For \(\varepsilon>0\), write
		\[
		L_\varepsilon v(x)
		=C_{N,s}\int_{|x-y|>\varepsilon}
		\frac{v(x)-v(y)}{|x-y|^{N+2s}}\,dy.
		\]
		Fix \(K\Subset\R^N\setminus\{0\}\), choose
		\(\theta\in(2s,\min\{2,2s+\beta\})\setminus\{1\}\), and take \(d>0\) so small
		that the closed \(2d\)-neighborhood of \(K\) avoids the origin. Local
		\(C^\theta\)-regularity gives
		\[
		|2u(x)-u(x+z)-u(x-z)|\le C_K|z|^\theta
		\qquad (x\in K,\ |z|<d).
		\]
		For \(\theta>1\), the linear terms cancel in this second difference.
		Symmetrizing in \(z\), we obtain, for \(0<\varepsilon<d\),
		\[
		\begin{aligned}
			&\int_{\varepsilon<|z|<d}
			\frac{u(x)-u(x+z)}{|z|^{N+2s}}\,dz\\
			&\qquad=\frac12\int_{\varepsilon<|z|<d}
			\frac{2u(x)-u(x+z)-u(x-z)}{|z|^{N+2s}}\,dz.
		\end{aligned}
		\]
		The integral over \(|z|<\varepsilon\) of the absolute value of the symmetrized
		integrand is bounded by \(C_K\varepsilon^{\theta-2s}\), uniformly for \(x\in K\).
		Away from the diagonal, the kernel satisfies
		\[
		\frac{1}{|x-y|^{N+2s}}
		\le\frac{C_{K,d}}{1+|y|^{N+2s}}
		\qquad (x\in K,\ |x-y|\ge d).
		\]
		Thus local integrability near the puncture and the tail at infinity are both
		controlled by \(u\in\Ls(\R^N)\). These estimates show that
		\(L_\varepsilon u\to(-\Delta)^su\) uniformly on \(K\). Each fixed truncated
		integral is continuous on \(K\) by dominated convergence, so
		\(F=\Hs[u]u^q=(-\Delta)^su\) is locally continuous away from the origin.
		
		Let now \(\varphi\in C_c^\infty(\R^N\setminus\{0\})\), and take \(K\) to contain
		its support. For each fixed \(\varepsilon>0\), the integrals below are absolutely convergent
		by the compact support of \(\varphi\), local integrability of \(u\), and its
		\(\Ls\) tail. Fubini's theorem and interchange of \(x\) and \(y\) give
		\[
		\begin{aligned}
			\int_{\R^N}\varphi L_\varepsilon u\,dx
			&=\frac{C_{N,s}}2\iint_{|x-y|>\varepsilon}
			\frac{(u(x)-u(y))(\varphi(x)-\varphi(y))}{|x-y|^{N+2s}}\,dx\,dy\\
			&=\int_{\R^N}uL_\varepsilon\varphi\,dx.
		\end{aligned}
		\]
		Near the diagonal, the numerator in the double integral is bounded by
		\(C|x-y|^{1+\min\{\theta,1\}}\), whose exponent is greater than \(2s\).
		The preceding tail estimate controls the remaining region.
		Moreover, the second difference estimate for the smooth test function and
		its compact support give
		\[
		|L_\varepsilon\varphi(x)|
		\le\frac{C_\varphi}{1+|x|^{N+2s}}
		\qquad (0<\varepsilon<1).
		\]
		Passing to the limit, using uniform convergence on \(\supp\varphi\) on the
		left and dominated convergence on the right, yields
		\[
		\int_{\R^N}u(-\Delta)^s\varphi\,dx
		=\int_{\R^N}(-\Delta)^su\,\varphi\,dx
		=\int_{\R^N}F\varphi\,dx.
		\]
		Finally, \(F\ge0\) and \eqref{eq:source-condition} imply
		\(F\in L^1_{\loc}(\R^N)\).
		
		Lemma~\ref{Lem2.2} shows that the function \(V\) defined by
		\eqref{eq:Riesz potential} is finite in the punctured space,
		belongs to \(\Ls(\R^N)\), and satisfies
		\((-\Delta)^sV=F\) in \(\mathcal D'(\R^N)\).
		
		Set
		\[
		h=u-V .
		\]
		By the preceding identity, \(u\) solves \((-\Delta)^s u=F\)
		distributionally in \(\R^N\setminus\{0\}\), while
		\((-\Delta)^s V=F\) in \(\R^N\). Therefore we have
		\[
		(-\Delta)^s h=0
		\quad \text{in }\R^N\setminus\{0\}.
		\]
		Also \(h\in\Ls(\R^N)\). Since \(V\ge0\), the bound
		\eqref{eq:fundamental-bound} gives the one-sided estimate
		\[
		h(x)\le u(x)\le C|x|^{-(N-2s)}
		\quad\text{for }0<|x|<r_0.
		\]
		Lemma~\ref{Lem2.3} therefore gives
		\[
		h=m\Phi_s+P
		\]
		as distributions, and hence almost everywhere in the punctured space,
		for some \(m\in\R\) and a polynomial \(P\) of degree less than \(2s\).
		Since \(u\), \(V\), \(\Phi_s\), and \(P\) are continuous there, the
		identity holds for every \(x\ne0\).
		
		We first prove \(m\ge0\). Suppose, to the contrary, that \(m<0\). Set
		\(\gamma=N-2s\) and
		\(A_\rho=\{x:\rho<|x|<2\rho\}\). The Riesz potential of a source satisfying
		\eqref{eq:source-condition} has vanishing averaged mass at the
		fundamental-solution scale near the origin:
		\begin{equation}
			\label{eq2.9}
			\frac{\rho^\gamma}{|A_\rho|}\int_{A_\rho}V(x)\,dx\to0
			\quad\text{as }\rho\to0^+ .
		\end{equation}
		Indeed, by Tonelli's theorem and scaling,
		\[
		\frac{\rho^\gamma}{|A_\rho|}\int_{A_\rho}V(x)\,dx
		=
		c_{N,s}\int_{\R^N}F(y)
		\left(
		\frac{\rho^\gamma}{|A_\rho|}\int_{A_\rho}\frac{dx}{|x-y|^\gamma}
		\right)dy.
		\]
		The factor in parentheses has the following uniform bound. If \(|y|\le4\rho\),
		then
		\[
		\frac{\rho^\gamma}{|A_\rho|}\int_{A_\rho}\frac{dx}{|x-y|^\gamma}
		\le C.
		\]
		If \(|y|>4\rho\), then
		\[
		\frac{\rho^\gamma}{|A_\rho|}\int_{A_\rho}\frac{dx}{|x-y|^\gamma}
		\le C\rho^\gamma |y|^{-\gamma}.
		\]
		Thus, for \(0<\rho<1\), the factor is bounded by
		\(C(1+|y|^\gamma)^{-1}\). For each fixed \(y\ne0\), it tends to zero as
		\(\rho\to0^+\). Hence \eqref{eq2.9} follows from dominated convergence and
		\eqref{eq:source-condition}.
		
		For each \(\rho\), choose a point in \(A_\rho\) at which \(V\) does not
		exceed its annular average. Since \(|x|\le2\rho\) on \(A_\rho\),
		\eqref{eq2.9} yields a sequence \(x_n\to0\) such that
		\(|x_n|^\gamma V(x_n)\to0\). Along this sequence,
		\[
		|x_n|^\gamma u(x_n)
		=
		|x_n|^\gamma V(x_n)+m c_{N,s}+|x_n|^\gamma P(x_n)
		\to m c_{N,s}<0,
		\]
		which contradicts \(u>0\). Therefore \(m\ge0\). The following argument
		eliminates \(P\).
		
		If \(s>1/2\), write \(P(x)=a\cdot x+b\). Suppose \(a\ne0\). There are a
		cone \(\Gamma\) of positive solid angle and \(c>0\) such that
		\[
		a\cdot y\ge c|y|
		\quad\text{for }y\in\Gamma.
		\]
		Since \(V\ge0\), \(m\Phi_s\ge0\), and \(b\) is fixed, for all sufficiently
		large \(y\in\Gamma\),
		\[
		u(y)\ge\frac c2|y|.
		\]
		For a fixed \(x\ne0\), this implies
		\[
		\Hs[u](x)
		\ge
		C\int_{\Gamma\cap\{|y|>R\}}|y|^{p-\mu}\,dy
		=
		C_\Gamma\int_R^\infty r^{N-1+p-\mu}\,dr
		=+\infty,
		\]
		because \(p>0\) and \(\mu<N\). This contradicts the finiteness of the
		Hartree potential. Hence \(a=0\). For \(s\le1/2\), Lemma~\ref{Lem2.3}
		already gives \(P=b\).
		
		We next exclude a positive constant. If \(b>0\), then
		\(u\ge b\) in the whole punctured space, and
		\[
		\Hs[u](x)
		\ge
		b^p\int_{\R^N}|x-y|^{-\mu}\,dy
		=+\infty.
		\]
		Therefore \(b\le0\).
		
		To exclude \(b<0\), we prove the decay of the annular averages of \(V\). For
		\[
		A_R=\{x:R<|x|<2R\},
		\]
		scaling and the local integrability of the Riesz kernel give
		\[
		\frac1{|A_R|}\int_{A_R}|x-y|^{-\gamma}\,dx
		\le
		C
		\begin{cases}
			R^{-\gamma},& |y|\le4R,\\
			|y|^{-\gamma},& |y|>4R.
		\end{cases}
		\]
		Tonelli's theorem yields
		\[
		\frac1{|A_R|}\int_{A_R}V(x)\,dx
		\le
		CR^{-\gamma}\int_{|y|\le4R}F(y)\,dy
		+
		C\int_{|y|>4R}F(y)|y|^{-\gamma}\,dy.
		\]
		The second term tends to zero by \eqref{eq:source-condition}. For the first term, the
		functions
		\[
		R^{-\gamma}\mathbf1_{\{|y|\le4R\}}
		\]
		converge pointwise to zero and are bounded, for \(R\ge1\), by
		\(C(1+|y|^\gamma)^{-1}\). Dominated convergence gives
		\begin{equation}
			\label{eq2.10}
			\frac1{|A_R|}\int_{A_R}V(x)\,dx\longrightarrow0
			\quad\text{as }R\to+\infty.
		\end{equation}
		
		If \(b<0\), positivity of \(u\) gives
		\[
		V(x)+m\Phi_s(x)>-b.
		\]
		Since \(\Phi_s(x)\to0\) at infinity, this implies
		\(V(x)\ge-b/2>0\) for every sufficiently large \(|x|\), in contradiction
		with \eqref{eq2.10}. Thus \(b=0\), and \eqref{eq:decomposition} follows. Applying
		\((-\Delta)^s\) to \eqref{eq:decomposition} in distributions proves the final
		assertion.
	\end{proof}

	\section{Moving spheres and radial symmetry}
	\label{sec:moving}
	
	In this section we prove Theorem~\ref{Thm1.2}. We use the Riesz decomposition
	obtained in Theorem~\ref{Thm1.1} and the direct method of moving spheres.
	The argument follows the integral comparison method of
	\cite{ChenLiOu,LiMovingSpheres,ChenLiZhang2017} and the Hartree splitting
	used in \cite{DaiFangQin}, adapted to the present punctured setting.
	
	Set
	\[
	F=\Hs[u]u^q,
	\qquad
	\Phi(y)=\Phi_s(y)=c_{N,s}|y|^{-(N-2s)} .
	\]
	By Theorem~\ref{Thm1.1}, there exists a constant
	\[
	m\ge0
	\]
	such that
	\begin{equation}
		\label{eq3.1}
		u(y)
		=
		c_{N,s}
		\int_{\R^N}
		\frac{F(z)}{|y-z|^{N-2s}}\,dz
		+m\Phi(y),
		\qquad y\in\R^N\setminus\{0\}.
	\end{equation}
	The term \(m\Phi\) is nonnegative in the Kelvin comparison.
	
	Fix throughout this section a point
	\[
	x\in\R^N\setminus\{0\}.
	\]
	We write \(\gamma=N-2s\).
	For \(0<\lambda<|x|\) and \(y\ne x\), set
	\[
	y^\lambda
	=
	x+\frac{\lambda^2(y-x)}{|y-x|^2},
	\qquad
	u_\lambda(y)
	=
	\left(\frac{\lambda}{|y-x|}\right)^{N-2s}u(y^\lambda),
	\]
	and
	\[
	\Sigma_\lambda=\R^N\setminus B_\lambda(x),
	\qquad
	w_\lambda(y)=u(y)-u_\lambda(y).
	\]
	When the center needs to be displayed, we write
	\[
	y^{x,\lambda}=y^\lambda,\qquad
	u_{x,\lambda}=u_\lambda,\qquad
	\Sigma_{x,\lambda}=\Sigma_\lambda.
	\]
	We also set
	\[
	\Phi_\lambda(y)
	=
	\left(\frac{\lambda}{|y-x|}\right)^{N-2s}\Phi(y^\lambda).
	\]
	The Kelvin transform is not defined at the point
	\[
	x-\frac{\lambda^2}{|x|^2}x,
	\]
	whose inversion is the puncture. We assign arbitrary values to transformed
	quantities at this single point. Its distance from \(x\) is
	\(\lambda^2/|x|<\lambda\), so it lies strictly inside \(B_\lambda(x)\) and
	does not belong to \(\Sigma_\lambda\). All pointwise statements below exclude
	the puncture and this exceptional point; removing these points does not affect
	the integral identities.
	
	The two Kelvin weights are
	\[
	\sigma_p=2N-\mu-(N-2s)p,
	\qquad
	\sigma_q=N+2s-\mu-(N-2s)q.
	\]
	By \eqref{eq1.3},
	\[
	\sigma_p\ge0,
	\qquad
	\sigma_q\ge0.
	\]
	
	The elementary inversion identities used below are
	\[
	|y^\lambda-z^\lambda|
	=
	\frac{\lambda^2|y-z|}{|y-x||z-x|},
	\qquad
	dz^\lambda=
	\left(\frac{\lambda}{|z-x|}\right)^{2N}dz.
	\]
	Since
	\[
	u(\xi^\lambda)
	=
	\left(\frac{|\xi-x|}{\lambda}\right)^\gamma u_\lambda(\xi),
	\]
	changing variables \(\eta=\xi^\lambda\) in
	\(\Hs[u](z^\lambda)\) gives
	\begin{align*}
		\Hs[u](z^\lambda)
		&=
		\int_{\R^N}
		\frac{u^p(\xi^\lambda)}
		{|z^\lambda-\xi^\lambda|^\mu}
		\left(\frac{\lambda}{|\xi-x|}\right)^{2N}d\xi\\
		&=
		\left(\frac{|z-x|}{\lambda}\right)^\mu
		\int_{\R^N}
		\frac{u_\lambda^p(\xi)}{|z-\xi|^\mu}
		\left(\frac{\lambda}{|\xi-x|}\right)^{
			2N-\mu-\gamma p}d\xi .
	\end{align*}
	Thus
	\[
	\sigma_p=2N-\mu-\gamma p
	\]
	and
	\[
	\mathcal H_{x,\lambda}[u](z)
	=
	\int_{\R^N}
	\frac{u_\lambda^p(\xi)}{|z-\xi|^\mu}
	\left(\frac{\lambda}{|\xi-x|}\right)^{\sigma_p}d\xi .
	\]
	Equivalently,
	\begin{equation}
		\label{eq3.2}
		\mathcal H_{x,\lambda}[u](z)
		=
		\left(\frac{\lambda}{|z-x|}\right)^\mu
		\Hs[u](z^\lambda).
	\end{equation}
	Moreover,
	\begin{align*}
		\left(\frac{\lambda}{|z-x|}\right)^{N+2s}F(z^\lambda)
		&=
		\left(\frac{\lambda}{|z-x|}\right)^{
			N+2s-\mu-\gamma q}
		\mathcal H_{x,\lambda}[u](z)u_\lambda^q(z)\\
		&=
		\left(\frac{\lambda}{|z-x|}\right)^{\sigma_q}
		\mathcal H_{x,\lambda}[u](z)u_\lambda^q(z),
	\end{align*}
	where
	\[
	\sigma_q=N+2s-\mu-\gamma q.
	\]
	Applying the inversion to the Riesz potential then yields
	\[
	u_\lambda(y)
	=
	c_{N,s}
	\int_{\R^N}
	\frac{\mathcal H_{x,\lambda}[u](z)u_\lambda^q(z)}
	{|y-z|^\gamma}
	\left(\frac{\lambda}{|z-x|}\right)^{\sigma_q}dz
	+m\Phi_\lambda(y).
	\]
	Define the transformed source by
	\begin{equation}
		\label{eq3.3}
		F_\lambda(z)
		=
		\left(\frac{\lambda}{|z-x|}\right)^{N+2s}F(z^\lambda).
	\end{equation}
	
	We repeatedly use the local boundedness of the Hartree potential away from the origin. Since
	\(F=(-\Delta)^s u\) is locally bounded in \(\R^N\setminus\{0\}\) and
	\(u>0\) is continuous, \(u\) is bounded away from zero on compact subsets of
	\(\R^N\setminus\{0\}\). Hence
	\[
	\Hs[u]=\frac{F}{u^q}
	\]
	is locally bounded there. The transformed potential
	\(\mathcal H_{x,\lambda}[u]\) is locally bounded on reflected compact sets by
	\eqref{eq3.2}.
	
	\begin{lemma}
		\label{Lem3.1}
		For every \(y\in\Sigma_\lambda\setminus\{0\}\), one has
		\begin{equation}
			\label{eq3.4}
			w_\lambda(y)
			=
			c_{N,s}
			\int_{\Sigma_\lambda}
			K_{x,\lambda}(y,z)\left(F(z)-F_\lambda(z)\right)dz
			+
			R_\lambda(y),
		\end{equation}
		where
		\[
		F(z)=\Hs[u](z)u^q(z),
		\qquad
		F_\lambda(z)=\mathcal H_{x,\lambda}[u](z)u_\lambda^q(z)
		\left(\frac{\lambda}{|z-x|}\right)^{\sigma_q},
		\]
		\begin{equation}
			\label{eq3.5}
			K_{x,\lambda}(y,z)
			=
			\frac{1}{|y-z|^{N-2s}}
			-
			\left(\frac{\lambda}{|y-x|}\right)^{N-2s}
			\frac{1}{|y^\lambda-z|^{N-2s}},
		\end{equation}
		and
		\[
		R_\lambda(y)
		=
		m\bigl(\Phi(y)-\Phi_\lambda(y)\bigr).
		\]
		Moreover,
		\begin{equation}
			\label{eq3.6}
			K_{x,\lambda}(y,z)>0
			\quad \text{if } |y-x|>\lambda,
			\quad |z-x|>\lambda,
		\end{equation}
		and
		\begin{equation}
			\label{eq3.7}
			R_\lambda(y)\ge0
			\quad \text{for } y\in\Sigma_\lambda\setminus\{0\}.
		\end{equation}
		If
		\[
		u_\lambda\le u
		\quad \text{in }\Sigma_\lambda\setminus\{0\},
		\]
		then
		\begin{equation}
			\label{eq3.8}
			F_\lambda(z)\le F(z)
			\quad \text{for every }z\in\Sigma_\lambda\setminus\{0\}.
		\end{equation}
	\end{lemma}
	
	\begin{proof}
		Apply the inversion to the Riesz potential part of \eqref{eq3.1}. Splitting the integral over \(B_\lambda(x)\) and \(\Sigma_\lambda\), and using \(z=\zeta^\lambda\) in the inner part, gives
		\[
		\left(\frac{\lambda}{|y-x|}\right)^{N-2s}
		\int_{\R^N}\frac{F(z)}{|y^\lambda-z|^{N-2s}}\,dz
		=
		\int_{\R^N}\frac{F_\lambda(z)}{|y-z|^{N-2s}}\,dz,
		\]
		with \(F_\lambda\) as above. Subtracting this identity, together with the transformed term \(m\Phi_\lambda\), from \eqref{eq3.1} yields \eqref{eq3.4} after the same decomposition into the interior and exterior regions. Formula \eqref{eq3.5} is the reflected Riesz kernel. The positivity \eqref{eq3.6} follows from
		\[
		|y-z|<\frac{|y-x|}{\lambda}|y^\lambda-z|,
		\qquad y,z\in\Sigma_\lambda,
		\]
		with strict inequality when both points are outside \(\partial B_\lambda(x)\).
		
		We prove \eqref{eq3.7}. For the fundamental solution part, observe that
		\[
		|y-x|^2|y^\lambda|^2-\lambda^2|y|^2
		=
		(|x|^2-\lambda^2)(|y-x|^2-\lambda^2)\ge0
		\]
		whenever \(0<\lambda<|x|\) and
		\(y\in\Sigma_\lambda\setminus\{0\}\). The first factor is positive because
		\(\lambda<|x|\), and the second is nonnegative because
		\(|y-x|\ge\lambda\). Hence
		\[
		|y-x|\,|y^\lambda|\ge \lambda |y|,
		\]
		which is equivalent to
		\[
		\Phi(y)\ge \Phi_\lambda(y).
		\]
		Since \(m\ge0\), this proves \(R_\lambda\ge0\).
		The comparison is strict when \(|y-x|>\lambda\); consequently
		\(R_\lambda(y)>0\) in the open exterior whenever \(m>0\).
		
		It remains to prove \eqref{eq3.8}. Since the inversion maps \(B_\lambda(x)\) onto \(\Sigma_\lambda\), a direct change of variables gives, for \(z\in\Sigma_\lambda\),
		\begin{align}
			\label{eq3.9}
			\Hs[u](z)
			&=\int_{\Sigma_\lambda}\frac{u^p(\xi)}{|z-\xi|^\mu}\,d\xi
			+\int_{\Sigma_\lambda}
			\left(\frac{\lambda}{|z-x|}\right)^\mu
			\frac{u_\lambda^p(\xi)}{|z^\lambda-\xi|^\mu}
			\left(\frac{\lambda}{|\xi-x|}\right)^{\sigma_p}d\xi,\\
			\label{eq3.10}
			\mathcal H_{x,\lambda}[u](z)
			&=\int_{\Sigma_\lambda}
			\frac{u_\lambda^p(\xi)}{|z-\xi|^\mu}
			\left(\frac{\lambda}{|\xi-x|}\right)^{\sigma_p}d\xi
			+\int_{\Sigma_\lambda}
			\left(\frac{\lambda}{|z-x|}\right)^\mu
			\frac{u^p(\xi)}{|z^\lambda-\xi|^\mu}d\xi .
		\end{align}
		Subtracting \eqref{eq3.10} from \eqref{eq3.9} yields
		\begin{align}
			\label{eq3.11}
			\Hs[u](z)-\mathcal H_{x,\lambda}[u](z)
			=\int_{\Sigma_\lambda}
			\left[
			\frac{1}{|z-\xi|^\mu}
			-
			\left(\frac{\lambda}{|z-x|}\right)^\mu
			\frac{1}{|z^\lambda-\xi|^\mu}
			\right]
			\left[
			u^p(\xi)-u_\lambda^p(\xi)
			\left(\frac{\lambda}{|\xi-x|}\right)^{\sigma_p}
			\right]d\xi .
		\end{align}
		For \(z,\xi\in\Sigma_\lambda\), the first bracket is nonnegative. If
		\(u_\lambda\le u\) in \(\Sigma_\lambda\setminus\{0\}\), then
		\[
		u_\lambda^p(\xi)
		\left(\frac{\lambda}{|\xi-x|}\right)^{\sigma_p}
		\le u_\lambda^p(\xi)\le u^p(\xi),
		\]
		because \(p>0\) and \(\sigma_p\ge0\). Thus
		\[
		\mathcal H_{x,\lambda}[u](z)\le \Hs[u](z).
		\]
		Similarly, \(q>0\) and \(\sigma_q\ge0\) give
		\[
		F_\lambda(z)
		\le
		\Hs[u](z)u_\lambda^q(z)
		\left(\frac{\lambda}{|z-x|}\right)^{\sigma_q}
		\le \Hs[u](z)u^q(z)=F(z),
		\]
		which proves \eqref{eq3.8}. This argument also covers the zero-weight
		endpoints.
	\end{proof}
	
	\begin{lemma}
		\label{Lem3.2}
		Let \(0<\lambda<|x|\). Let
		\(Q\subset\R^N\setminus\{0\}\) be compact and let
		\(\Omega\subset Q\cap\Sigma_\lambda\) be measurable. Assume that
		\[
		w_\lambda\ge0
		\quad\text{in }\Sigma_\lambda\setminus\Omega .
		\]
		Set
		\[
		W=(u_\lambda-u)^+,
		\qquad
		\Omega^-=\{z\in\Sigma_\lambda:w_\lambda(z)<0\}.
		\]
		Assume that \(u\), \(u_\lambda\), \(\Hs[u]\), and
		\(\mathcal H_{x,\lambda}[u]\) are bounded on
		\(Q\cap\Sigma_\lambda\). Then
		\(\Omega^-\subset\Omega\), and there exists \(C_Q>0\), depending only on
		these bounds and on \(p,q\), such that for almost every
		\(z\in\Sigma_\lambda\),
		\begin{equation}
			\label{eq3.12}
			\bigl(F_\lambda(z)-F(z)\bigr)^+
			\le
			C_Q W(z)
			+
			C_Q u^q(z)
			\int_{\Omega^-} \frac{W(\xi)}{|z-\xi|^\mu}\,d\xi .
		\end{equation}
	\end{lemma}
	
	\begin{proof}
		Write
		\[
		a_\lambda(z)=\left(\frac{\lambda}{|z-x|}\right)^{\sigma_q},
		\qquad
		b_\lambda(\xi)=\left(\frac{\lambda}{|\xi-x|}\right)^{\sigma_p}.
		\]
		Since \(z,\xi\in\Sigma_\lambda\) and \(\sigma_p,\sigma_q\ge0\), one has
		\(0<a_\lambda,b_\lambda\le1\). Hence
		\[
		F_\lambda-F
		=
		a_\lambda \mathcal H_{x,\lambda}[u]u_\lambda^q
		-
		\Hs[u]u^q
		\]
		satisfies
		\begin{equation}
			\label{eq3.13}
			F_\lambda-F
			\le
			a_\lambda \mathcal H_{x,\lambda}[u](u_\lambda^q-u^q)_+
			+
			u^q\bigl(\mathcal H_{x,\lambda}[u]-\Hs[u]\bigr)^+ .
		\end{equation}
		Since \(w_\lambda\ge0\) in \(\Sigma_\lambda\setminus\Omega\), one has
		\(\Omega^-=\{W>0\}\subset\Omega\). For \(r\ge1\) and \(0\le A,B\le M\),
		\[
		(A^r-B^r)^+\le rM^{r-1}(A-B)^+.
		\]
		This includes the endpoint \(r=1\). Since
		\(\mathcal H_{x,\lambda}[u]\) is bounded on \(Q\), the first term in
		\eqref{eq3.13} vanishes outside \(\Omega^-\) and satisfies
		\begin{equation}
			\label{eq3.14}
			a_\lambda \mathcal H_{x,\lambda}[u](u_\lambda^q-u^q)_+
			\le
			C_Q W .
		\end{equation}
		
		Next, by \eqref{eq3.11},
		\[
		\mathcal H_{x,\lambda}[u](z)-\Hs[u](z)
		=
		\int_{\Sigma_\lambda}
		A_\lambda(z,\xi)
		\left[
		b_\lambda(\xi)u_\lambda^p(\xi)-u^p(\xi)
		\right]d\xi ,
		\]
		where
		\[
		A_\lambda(z,\xi)
		=
		\frac1{|z-\xi|^\mu}
		-
		\left(\frac{\lambda}{|z-x|}\right)^\mu
		\frac1{|z^\lambda-\xi|^\mu}
		\ge0 .
		\]
		If \(\xi\in\Sigma_\lambda\setminus\Omega^-\), then
		\(u_\lambda(\xi)\le u(\xi)\),
		and hence
		\[
		b_\lambda(\xi)u_\lambda^p(\xi)-u^p(\xi)\le0 .
		\]
		Thus only the density on \(\Omega^-\) can generate a positive potential
		defect. Since \(p\ge1\) and \(u,u_\lambda\) are bounded on \(Q\),
		\[
		b_\lambda(\xi)u_\lambda^p(\xi)-u^p(\xi)
		\le
		u_\lambda^p(\xi)-u^p(\xi)
		\le
		C_Q W(\xi)
		\quad\text{on }\Omega^-.
		\]
		Using \(A_\lambda(z,\xi)\le |z-\xi|^{-\mu}\), we obtain
		\begin{equation}
			\label{eq3.15}
			\bigl(\mathcal H_{x,\lambda}[u](z)-\Hs[u](z)\bigr)^+
			\le
			C_Q\int_{\Omega^-} \frac{W(\xi)}{|z-\xi|^\mu}\,d\xi .
		\end{equation}
		The integral is finite: \(W\) is bounded on the compact set \(Q\), and
		\(\mu<N\). Combining \eqref{eq3.13}, \eqref{eq3.14}, and
		\eqref{eq3.15} gives \eqref{eq3.12}. Thus the generating density,
		not the resulting convolution, is localized to \(\Omega^-\).
	\end{proof}
	
	\begin{lemma}
		\label{Lem3.3}
		Let \(Q\) be a compact subset of \(\mathbb R^N\setminus\{0\}\). Then there exists a nondecreasing
		modulus \(\omega_Q(t)\), with \(\omega_Q(t)\to0\) as \(t\to0^+\), such that
		for every measurable set \(E\subset Q\),
		\begin{equation}
			\label{eq3.16}
			\sup_{y\in Q}
			\int_{\mathbb R^N}
			\frac{u^q(z)}{|y-z|^{N-2s}}
			\left(
			\int_E \frac{d\xi}{|z-\xi|^\mu}
			\right)dz
			\le
			\omega_Q(|E|).
		\end{equation}
	\end{lemma}
	
	\begin{proof}
		Choose a closed ball \(B\) contained in
		\(\mathbb R^N\setminus\{0\}\). Since \(u>0\),
		\[
		M_B=\int_B u^p(\eta)\,d\eta>0 .
		\]
		For all \(z\in\mathbb R^N\), this gives
		\[
		\Hs[u](z)
		\ge
		c(1+|z|)^{-\mu}.
		\]
		Hence
		\begin{equation}
			\label{eq3.17}
			u^q(z)(1+|z|)^{-\mu}
			\le
			C F(z).
		\end{equation}
		By \eqref{eq3.1} and the nonnegativity of \(m\Phi\),
		\[
		c_{N,s}\int_{\mathbb R^N}\frac{F(z)}{|y-z|^{N-2s}}\,dz
		\le u(y).
		\]
		Since \(Q\Subset\R^N\setminus\{0\}\) and \(u\in C_{\loc}(\R^N\setminus\{0\})\), the function \(u\) is bounded on \(Q\). Thus \eqref{eq3.17} gives, uniformly for \(y\in Q\),
		\begin{equation}
			\label{eq3.18}
			\int_{\mathbb R^N}
			\frac{u^q(z)(1+|z|)^{-\mu}}{|y-z|^{N-2s}}\,dz
			\le
			C_Q .
		\end{equation}
		
		Let
		\[
		J_E(z)=\int_E |z-\xi|^{-\mu}\,d\xi .
		\]
		Since \(\mu<N\), the integral of the radially decreasing kernel
		\(|\cdot|^{-\mu}\) over a set of prescribed measure is maximized by a ball
		centered at the kernel singularity. Integrating over that ball in polar
		coordinates gives
		\begin{equation}
			\label{eq3.19}
			J_E(z)\le C|E|^{\frac{N-\mu}{N}}
			\qquad\text{for all }z\in\mathbb R^N.
		\end{equation}
		If \(|B_r(z)|=|E|\), then
		\[
		\int_{E\setminus B_r(z)}|z-\xi|^{-\mu}\,d\xi
		\le r^{-\mu}|E\setminus B_r(z)|
		=r^{-\mu}|B_r(z)\setminus E|
		\le\int_{B_r(z)\setminus E}|z-\xi|^{-\mu}\,d\xi,
		\]
		which directly proves that the ball maximizes the integral.
		Because \(E\subset Q\) and \(Q\) is bounded, there is \(R_Q>0\) such that
		\begin{equation}
			\label{eq3.20}
			J_E(z)\le C_Q |E|(1+|z|)^{-\mu}
			\qquad\text{for } |z|\ge R_Q .
		\end{equation}
		On the bounded ball \(B_{R_Q}\), the lower bound
		\(\Hs[u]\ge c_Q>0\) follows again from the positivity of \(u\) on a fixed ball.
		Thus \(u^q\le C_Q F\) on \(B_{R_Q}\). Combining this with
		\eqref{eq3.1}, the nonnegativity of \(m\Phi\), and \eqref{eq3.19}, we get
		\[
		\sup_{y\in Q}
		\int_{B_{R_Q}}
		\frac{u^q(z)J_E(z)}{|y-z|^{N-2s}}\,dz
		\le
		C_Q |E|^{\frac{N-\mu}{N}} .
		\]
		On \(\mathbb R^N\setminus B_{R_Q}\), using \eqref{eq3.20} and
		\eqref{eq3.18}, we get
		\[
		\sup_{y\in Q}
		\int_{\mathbb R^N\setminus B_{R_Q}}
		\frac{u^q(z)J_E(z)}{|y-z|^{N-2s}}\,dz
		\le
		C_Q |E| .
		\]
		Therefore \eqref{eq3.16} holds with, for instance,
		\[
		\omega_Q(t)
		=
		C_Q\left(t^{\frac{N-\mu}{N}}+t\right).
		\]
	\end{proof}
	
	\begin{lemma}
		\label{Lem3.4}
		Let \(0<\lambda_0<\lambda_1<|x|\) and \(R>0\). There exists
		\(\ell_0>0\) such that the following holds. Let
		\(\lambda\in[\lambda_0,\lambda_1]\), and let
		\[
		\Omega\subset \Sigma_\lambda\cap B_R
		\]
		be measurable and contained in
		\[
		\{y:\lambda<|y-x|<\lambda+\ell\}
		\]
		for some \(0<\ell\le\ell_0\). If
		\[
		w_\lambda\ge0
		\quad \text{in }\Sigma_\lambda\setminus\Omega,
		\]
		then
		\[
		w_\lambda\ge0
		\quad \text{in }\Omega.
		\]
	\end{lemma}
	
	\begin{proof}
		Choose \(\ell_*>0\) such that
		\[
		\lambda_1+\ell_*<|x|.
		\]
		Set
		\[
		Q=\overline{B_R}\cap\{y:\lambda_0\le |y-x|\le \lambda_1+\ell_*\}.
		\]
		Then \(Q\) is a compact subset of \(\mathbb R^N\setminus\{0\}\), since
		\[
		|y|\ge |x|-(\lambda_1+\ell_*)>0
		\quad\text{for }y\in Q.
		\]
		If
		\(0<\ell\le\ell_*\), then \(\Omega\subset Q\).
		
		For \(y\in Q\cap\Sigma_\lambda\),
		\[
		|y^\lambda-x|=\frac{\lambda^2}{|y-x|}\le\lambda\le\lambda_1.
		\]
		Thus \(y^\lambda\in\overline{B_{\lambda_1}(x)}\), which is a compact
		subset of \(\mathbb R^N\setminus\{0\}\). Since
		\(\lambda/|y-x|\le1\), the functions \(u\), \(u_\lambda\), \(\Hs[u]\),
		and, by \eqref{eq3.2}, \(\mathcal H_{x,\lambda}[u]\) are uniformly bounded
		on \(Q\cap\Sigma_\lambda\), uniformly for
		\(\lambda\in[\lambda_0,\lambda_1]\).
		
		Suppose for contradiction that
		\[
		\Omega^-=\{y\in\Omega:w_\lambda(y)<0\}
		\]
		is nonempty. Set
		\[
		W=(u_\lambda-u)^+ .
		\]
		Then \(W=0\) on \(\Sigma_\lambda\setminus\Omega^-\). For \(y\in\Omega^-\), using
		\eqref{eq3.4} and the positivity of \(K_{x,\lambda}\), we have
		\begin{equation}
			\label{eq3.21}
			W(y)
			=
			-w_\lambda(y)
			\le
			c_{N,s}
			\int_{\Sigma_\lambda}
			K_{x,\lambda}(y,z)
			\bigl(F_\lambda(z)-F(z)\bigr)^+\,dz .
		\end{equation}
		By Lemma~\ref{Lem3.2},
		\[
		\bigl(F_\lambda(z)-F(z)\bigr)^+
		\le
		CW(z)
		+
		Cu^q(z)\int_{\Omega^-}\frac{W(\xi)}{|z-\xi|^\mu}\,d\xi .
		\]
		Since \(K_{x,\lambda}(y,z)\le |y-z|^{-(N-2s)}\), \eqref{eq3.21}
		gives
		\begin{equation}
			\label{eq3.22}
			\begin{aligned}
				W(y)
				&\le
				C\int_{\Omega^-}
				\frac{W(z)}{|y-z|^{N-2s}}\,dz                                      \\
				&\quad
				+
				C\int_{\Sigma_\lambda}
				\frac{u^q(z)}{|y-z|^{N-2s}}
				\left(
				\int_{\Omega^-}\frac{W(\xi)}{|z-\xi|^\mu}\,d\xi
				\right)dz .
			\end{aligned}
		\end{equation}
		Put
		\[
		M=\operatorname*{ess\,sup}_{\Omega^-}W.
		\]
		The first term is bounded by
		\[
		C|\Omega^-|^{\frac{2s}{N}}M .
		\]
		Indeed, the same rearrangement estimate, now with kernel
		\(|\cdot|^{-(N-2s)}\), gives
		\[
		\sup_{y\in\R^N}
		\int_{\Omega^-}|y-z|^{-(N-2s)}\,dz
		\le C|\Omega^-|^{2s/N}.
		\]
		For the second term, Lemma~\ref{Lem3.3} gives
		\[
		C\omega_Q(|\Omega^-|)
		M .
		\]
		Taking the essential supremum over \(y\in\Omega^-\), we obtain
		\begin{equation}
			\label{eq3.23}
			M
			\le
			C\left(
			|\Omega^-|^{\frac{2s}{N}}
			+
			\omega_Q(|\Omega^-|)
			\right)
			M .
		\end{equation}
		The constant \(C\) is independent of \(\lambda\), \(\ell\), and \(\Omega\).
		
		Since \(\Omega\) is contained in a shell of width \(\ell\) inside \(B_R\),
		\[
		|\Omega|
		\le
		|B_1(0)|\bigl((\lambda+\ell)^N-\lambda^N\bigr)
		\le
		|B_1(0)|\bigl((\lambda_1+\ell)^N-\lambda_1^N\bigr)
		\longrightarrow0
		\quad\text{as }\ell\to0^+
		\]
		uniformly for \(\lambda\in[\lambda_0,\lambda_1]\). Choose
		\(0<\ell_0\le\ell_*\) so small that the coefficient on the right-hand side of
		\eqref{eq3.23} is strictly less than \(1\). Then
		\[
		M=0.
		\]
		Since \(w_\lambda\) is continuous on \(Q\cap\Sigma_\lambda\), a nonempty negative set
		\(\Omega^-\) is open relative to \(Q\), has positive measure, and satisfies
		\(M>0\). We have reached a contradiction. Hence
		\[
		\Omega^-=\varnothing,
		\]
		and therefore \(w_\lambda\ge0\) in \(\Omega\).
	\end{proof}
	
	\begin{lemma}
		\label{Lem3.5}
		Let \(L>1\) be fixed. There exists \(\lambda_1>0\) such that, for every
		\(0<\lambda<\lambda_1\) and every measurable set
		\[
		\Omega\subset \{y:\lambda<|y-x|<L\lambda\},
		\]
		the condition
		\[
		w_\lambda\ge0
		\quad \text{in }\Sigma_\lambda\setminus\Omega,
		\]
		implies
		\[
		w_\lambda\ge0
		\quad \text{in }\Omega.
		\]
	\end{lemma}
	
	\begin{proof}
		Choose \(\bar\lambda>0\) such that
		\[
		2L\bar\lambda<|x|,
		\]
		and set
		\[
		Q=\overline{B_{2L\bar\lambda}(x)}.
		\]
		Then \(Q\Subset\mathbb R^N\setminus\{0\}\).
		All constants below are allowed to depend on \(Q,L,x,u\), but not on
		\(\lambda\in(0,\bar\lambda)\) or on \(\Omega\).
		
		Let \(0<\lambda<\bar\lambda\). Then
		\[
		\Omega\subset\{y:\lambda<|y-x|<L\lambda\}\subset Q.
		\]
		For \(y\in Q\cap\Sigma_\lambda\), one has
		\(y^\lambda\in\overline{B_\lambda(x)}\subset
		\overline{B_{\bar\lambda}(x)}\). Therefore
		\[
		u_\lambda(y)
		\le
		\sup_{\overline{B_{\bar\lambda}(x)}}u,
		\]
		and, by \eqref{eq3.2},
		\[
		\mathcal H_{x,\lambda}[u](y)
		\le
		\sup_{\overline{B_{\bar\lambda}(x)}}\Hs[u].
		\]
		Thus \(u\), \(u_\lambda\), \(\Hs[u]\), and
		\(\mathcal H_{x,\lambda}[u]\) are uniformly bounded on
		\(Q\cap\Sigma_\lambda\).
		
		Suppose for contradiction that
		\[
		\Omega^-=\{y\in\Omega:w_\lambda(y)<0\}
		\]
		is nonempty, and set
		\[
		W=(u_\lambda-u)^+ .
		\]
		Then \(W=0\) on \(\Sigma_\lambda\setminus\Omega^-\). Repeating the estimate in
		\eqref{eq3.21}--\eqref{eq3.22}, using Lemma~\ref{Lem3.2} and
		Lemma~\ref{Lem3.3}, gives
		\[
		M
		\le
		C\left(
		|\Omega^-|^{\frac{2s}{N}}
		+
		\omega_Q(|\Omega^-|)
		\right)
		M,
		\qquad
		M=\operatorname*{ess\,sup}_{\Omega^-}W.
		\]
		Since
		\[
		|\Omega^-|\le|\Omega|
		\le
		|B_{L\lambda}(x)\setminus B_\lambda(x)|
		=
		|B_1(0)|(L^N-1)\lambda^N,
		\]
		we have \(|\Omega^-|\to0\) as \(\lambda\to0^+\). Choose
		\(0<\lambda_1\le\bar\lambda\) so small that, whenever
		\(0<\lambda<\lambda_1\), the coefficient on the right-hand side is strictly
		less than \(1\). Then
		\[
		M=0.
		\]
		As in Lemma~\ref{Lem3.4}, continuity of \(w_\lambda\) on
		\(Q\cap\Sigma_\lambda\) implies
		that a nonempty \(\Omega^-\) has positive measure and \(M>0\). This is a
		contradiction. Therefore
		\[
		\Omega^-=\varnothing,
		\]
		and hence \(w_\lambda\ge0\) in \(\Omega\).
	\end{proof}
	
	\begin{lemma}
		\label{Lem3.6}
		For every \(x\in\R^N\setminus\{0\}\), there exists \(\lambda_0(x)>0\) such that
		\[
		u_{x,\lambda}(y)\le u(y)
		\quad \text{for all }0<\lambda<\lambda_0(x),\ y\in\Sigma_{x,\lambda}\setminus\{0\}.
		\]
	\end{lemma}
	
	\begin{proof}
		Choose \(r_0>0\) such that
		\[
		\overline{B_{4r_0}(x)}\subset\R^N\setminus\{0\}.
		\]
		Since \(u\) is positive and continuous there, there exist constants
		\(m_x,M_x>0\) such that
		\[
		0<m_x\le u\le M_x
		\quad \text{in }B_{4r_0}(x).
		\]
		Choose \(L>2\) such that
		\[
		M_xL^{-(N-2s)}\le \frac{m_x}{2}.
		\]
		If
		\[
		L\lambda\le |y-x|\le 2r_0,
		\]
		then \(y^\lambda\in B_\lambda(x)\subset B_{r_0}(x)\), and hence
		\[
		u_\lambda(y)
		\le
		M_x\left(\frac{\lambda}{|y-x|}\right)^{N-2s}
		\le \frac{m_x}{2}
		\le u(y).
		\]
		
		Next choose \(\delta_0>0\) so small that
		\[
		\overline{B_{\delta_0}(0)}
		\cap\overline{B_{2r_0}(x)}=\varnothing
		\]
		and \(u(y)\ge1\) for \(0<|y|\le\delta_0\), using \eqref{eq:blowup}.
		If \(0<|y|\le\delta_0\), then \(|y-x|\) is bounded below by a positive
		constant. Hence there exists \(\lambda_a>0\) such that, for every
		\(0<\lambda\le\lambda_a\), the points \(y^\lambda\) lie in
		\(\overline{B_{r_0}(x)}\), uniformly in \(y\), and
		\[
		\left(\frac{\lambda}{|y-x|}\right)^{N-2s}\le C\lambda^{N-2s}.
		\]
		After decreasing \(\lambda_a\) so that
		\(CM_x\lambda_a^{N-2s}\le1\), we have
		\[
		u_\lambda(y)\le u(y)
		\quad\text{for }0<|y|\le\delta_0,\quad 0<\lambda\le\lambda_a.
		\]
		
		Let \(R>1\) be large. On the compact set
		\[
		K_{\delta_0,R}
		=
		\{y:\delta_0\le |y|\le R,\ |y-x|\ge 2r_0\},
		\]
		we have
		\[
		m_{\delta_0,R}:=\min_{K_{\delta_0,R}}u>0.
		\]
		For \(y\in K_{\delta_0,R}\) and \(\lambda\) small,
		\(y^\lambda\in\overline{B_{r_0}(x)}\);
		hence
		\[
		u_\lambda(y)
		\le
		C\lambda^{N-2s}.
		\]
		After decreasing \(\lambda\), this gives
		\[
		u_\lambda(y)\le \frac12m_{\delta_0,R}\le u(y)
		\quad\text{on }K_{\delta_0,R}.
		\]
		
		It remains to treat the far field. Choose a closed ball
		\[
		B_*\subset\R^N\setminus\{0\}
		\]
		such that
		\[
		\int_{B_*}F(z)\,dz>0 .
		\]
		Then, by \eqref{eq3.1}, for \(|y|\) large,
		\[
		u(y)\ge c|y|^{-(N-2s)}.
		\]
		After decreasing \(\lambda_a\) if necessary,
		\(y^\lambda\in\overline{B_{r_0}(x)}\) whenever \(|y|\) is large and
		\(0<\lambda\le\lambda_a\). Hence
		\[
		u_\lambda(y)
		\le
		C\lambda^{N-2s}|y|^{-(N-2s)}.
		\]
		After decreasing the common upper bound for \(\lambda\), we obtain
		\[
		u_\lambda(y)\le u(y)
		\quad\text{for }|y|\ge R.
		\]
		
		The constants have been chosen in the order
		\[
		r_0,\quad L,\quad \delta_0,\quad R,
		\]
		followed by a common upper bound for \(\lambda\).
		Combining the preceding estimates, any point at which \(w_\lambda<0\)
		must lie in the thin annulus
		\[
		\Omega_\lambda=\{y:\lambda<|y-x|<L\lambda\}.
		\]
		For sufficiently small \(\lambda\), Lemma~\ref{Lem3.5} applies to
		\(\Omega_\lambda\), and gives
		\[
		w_\lambda\ge0
		\quad\text{in }\Omega_\lambda.
		\]
		Therefore \(w_\lambda\ge0\) in \(\Sigma_\lambda\) for all sufficiently small
		\(\lambda>0\).
	\end{proof}
	
	\begin{lemma}
		\label{Lem3.7}
		Let \(0<\lambda_*<|x|\), and assume that
		\[
		w_{\lambda_*}\ge0
		\quad \text{in }\Sigma_{\lambda_*},
		\]
		with strict inequality in
		\(\Sigma_{\lambda_*}\setminus\partial B_{\lambda_*}(x)\). Then there exist
		\(R>0\) and \(\delta>0\) such that
		\[
		w_\lambda(y)\ge0
		\quad \text{for } |y|\ge R,
		\quad \lambda\in[\lambda_*,\lambda_*+\delta].
		\]
	\end{lemma}
	
	\begin{proof}
		We first obtain a positive lower bound for \(w_{\lambda_*}\) at infinity. By
		Lemma~\ref{Lem3.1}, \(F_{\lambda_*}\le F\). If \(m>0\), the claimed bound
		therefore follows directly from the nonnegative remainder
		\(R_{\lambda_*}\). Indeed,
		\[
		R_{\lambda_*}(y)\ge c|y|^{-(N-2s)}
		\quad\text{for } |y| \text{ large},
		\]
		because
		\[
		|y|^{N-2s}\bigl(\Phi(y)-\Phi_{\lambda_*}(y)\bigr)\to
		c_{N,s}\left[
		1-\left(\frac{\lambda_*}{|x|}\right)^{N-2s}
		\right]>0.
		\]
		
		It remains to consider the case \(m=0\). Then Lemma~\ref{Lem3.1} gives
		\(F_{\lambda_*}\le F\) in \(\Sigma_{\lambda_*}\). If
		\(F-F_{\lambda_*}=0\) almost everywhere there, then \eqref{eq3.4} would
		give \(w_{\lambda_*}\equiv0\), contrary to the assumed strict positivity.
		Thus \(F-F_{\lambda_*}\) is positive on a set of positive measure.
		Moreover, \(F\) is locally continuous away from the origin, because
		\(F=(-\Delta)^s u\) and
		\(u\in C_{\loc}^{2s+\beta}(\R^N\setminus\{0\})\). By \eqref{eq3.3}, the
		same is true for \(F_{\lambda_*}\) on compact sets avoiding the origin and
		the boundary sphere. Hence there exist a compact set
		\[
		E\Subset
		\{z:|z-x|>\lambda_*,\ z\ne0\}
		\]
		of positive measure and a number \(c_E>0\) such that
		\[
		F(z)-F_{\lambda_*}(z)\ge c_E
		\quad\text{for }z\in E.
		\]
		Since \(E\) is compactly contained in \(\Sigma_{\lambda_*}\), one has,
		uniformly for \(z\in E\),
		\[
		|y|^{N-2s}K_{x,\lambda_*}(y,z)
		\longrightarrow
		1-\left(\frac{\lambda_*}{|x-z|}\right)^{N-2s}>0
		\quad\text{as }|y|\to+\infty.
		\]
		The positive limiting function has a positive minimum on \(E\). Hence
		\[
		K_{x,\lambda_*}(y,z)\ge c|y|^{-(N-2s)}
		\quad\text{for }z\in E,\ |y|\text{ sufficiently large},
		\]
		and therefore \eqref{eq3.4} gives the same lower bound for \(w_{\lambda_*}\). Thus, in all cases, there exist \(c_0>0\) and \(R_0>0\) such that
		\begin{equation}
			\label{eq3.24}
			w_{\lambda_*}(y)\ge c_0|y|^{-(N-2s)}
			\quad\text{for } |y|\ge R_0 .
		\end{equation}
		
		We compare \(u_\lambda\) with \(u_{\lambda_*}\) at infinity. Choose
		\(\delta_0>0\) such that \(\lambda_*+\delta_0<|x|\). For
		\(\lambda\in[\lambda_*,\lambda_*+\delta_0]\), set
		\[
		A_\lambda(y)=|y|^\gamma u_\lambda(y)
		=
		\left(\frac{\lambda |y|}{|y-x|}\right)^\gamma
		u\left(x+\frac{\lambda^2(y-x)}{|y-x|^2}\right).
		\]
		The convergence
		\[
		A_\lambda(y)\longrightarrow a_\lambda:=\lambda^\gamma u(x)
		\quad\text{as }|y|\to+\infty
		\]
		is uniform in this interval. Indeed,
		\[
		\sup_\lambda|y^\lambda-x|
		\le
		\frac{(\lambda_*+\delta_0)^2}{|y-x|}
		\longrightarrow0,
		\qquad
		\frac{|y|}{|y-x|}\longrightarrow1,
		\]
		and \(u\) is uniformly continuous near \(x\). Choose first
		\(0<\delta\le\delta_0\) so small that
		\[
		\bigl|a_\lambda-a_{\lambda_*}\bigr|
		\le \frac{c_0}{8}
		\quad\text{for } \lambda\in[\lambda_*,\lambda_*+\delta].
		\]
		Then increase \(R_0\), if necessary, so that
		\[
		\sup_{\lambda\in[\lambda_*,\lambda_*+\delta]}
		|A_\lambda(y)-a_\lambda|
		\le\frac{c_0}{8}
		\quad\text{for }|y|\ge R_0.
		\]
		The same estimate includes \(\lambda=\lambda_*\). Therefore
		\[
		|u_\lambda(y)-u_{\lambda_*}(y)|
		\le
		\frac{3c_0}{8}|y|^{-\gamma}
		\]
		for every \(|y|\ge R_0\) and every \(\lambda\in[\lambda_*,\lambda_*+\delta]\). Combining this with \eqref{eq3.24}, we get
		\[
		w_\lambda(y)
		=
		w_{\lambda_*}(y)+u_{\lambda_*}(y)-u_\lambda(y)
		\ge
		\frac{5c_0}{8}|y|^{-\gamma}
		\ge
		\frac{c_0}{2}|y|^{-\gamma}
		\]
		for \(|y|\ge R_0\) and \(\lambda\in[\lambda_*,\lambda_*+\delta]\). This proves the lemma.
	\end{proof}
	
	\begin{lemma}
		\label{Lem3.8}
		Let \(0<\lambda_*<|x|\). If
		\[
		u_\rho\le u
		\quad \text{in }\Sigma_\rho
		\quad \text{for every }0<\rho<\lambda_*,
		\]
		then there exists \(\delta>0\) such that
		\[
		u_\lambda\le u
		\quad \text{in }\Sigma_\lambda
		\quad \text{for every }0<\lambda<\lambda_*+\delta.
		\]
	\end{lemma}
	
	\begin{proof}
		For \(j\ge1\), set
		\[
		K_j
		=
		\left\{
		y\in\R^N:
		j^{-1}\le|y|\le j,\quad
		|y-x|\ge\lambda_*+j^{-1}
		\right\}.
		\]
		As \(\rho\uparrow\lambda_*\), the reflected sets
		\(\{y^\rho:y\in K_j\}\) remain in a fixed compact subset of
		\(\R^N\setminus\{0\}\). The continuity of \(u\) therefore gives
		\[
		u_\rho\longrightarrow u_{\lambda_*}
		\quad\text{uniformly on }K_j.
		\]
		The assumed inequalities imply \(w_{\lambda_*}\ge0\) on every \(K_j\).
		Letting \(j\to\infty\) and observing that inversion fixes the boundary
		sphere, we obtain
		\[
		w_{\lambda_*}\ge0
		\quad \text{in }\Sigma_{\lambda_*}.
		\]
		The inequality is strict in
		\(\Sigma_{\lambda_*}\setminus\partial B_{\lambda_*}(x)\). Indeed, suppose that equality holds at an interior point \(y_0\). By Lemma~\ref{Lem3.1}, we have
		\[
		F_{\lambda_*}\le F
		\quad\text{in }\Sigma_{\lambda_*}.
		\]
		Since \(K_{x,\lambda_*}>0\) and \(R_{\lambda_*}\ge0\), formula \eqref{eq3.4} implies both
		\[
		F-F_{\lambda_*}=0
		\quad\text{a.e. in }\Sigma_{\lambda_*}
		\]
		and
		\[
		R_{\lambda_*}(y_0)=0.
		\]
		If \(m>0\), then \(R_{\lambda_*}(y)>0\) for every interior point \(y\in\Sigma_{\lambda_*}\setminus\partial B_{\lambda_*}(x)\), a contradiction. Hence \(m=0\). In that case \eqref{eq3.4} gives \(w_{\lambda_*}\equiv0\) in \(\Sigma_{\lambda_*}\). This is impossible because, as \(y\to0\) through \(\Sigma_{\lambda_*}\setminus\{0\}\), one has \(u(y)\to+\infty\), whereas \(u_{\lambda_*}(y)\) remains bounded. Thus \(w_{\lambda_*}(y)\to+\infty\), contradicting \(w_{\lambda_*}\equiv0\).
		
		Choose \(\delta_0>0\) such that
		\[
		\lambda_*+\delta_0<|x|.
		\]
		By Lemma~\ref{Lem3.7}, after decreasing \(\delta_0\) if necessary, there is
		\(R>0\) such that
		\[
		w_\lambda(y)\ge0
		\quad\text{for } |y|\ge R,
		\quad \lambda\in[\lambda_*,\lambda_*+\delta_0].
		\]
		
		Near the origin, the reflected points
		\[
		y^\lambda=x+\frac{\lambda^2(y-x)}{|y-x|^2},
		\qquad
		\lambda\in[\lambda_*,\lambda_*+\delta_0],
		\]
		remain in a fixed compact subset of \(\R^N\setminus\{0\}\), because
		\(\lambda_*+\delta_0<|x|\). More precisely, their limiting distance from
		the origin is bounded below uniformly by
		\[
		|x|-\frac{(\lambda_*+\delta_0)^2}{|x|}>0.
		\]
		The Kelvin prefactor is uniformly bounded there.
		Thus \(u_\lambda(y)\) is uniformly bounded, whereas \(u(y)\to+\infty\).
		Hence there exists
		\(\rho_0>0\) such that
		\[
		w_\lambda(y)\ge0
		\quad\text{for }0<|y|<\rho_0,
		\quad \lambda\in[\lambda_*,\lambda_*+\delta_0].
		\]
		
		Apply Lemma~\ref{Lem3.4} with, for example,
		\[
		\lambda_0=\frac{\lambda_*}{2},
		\qquad
		\lambda_1=\lambda_*+\delta_0,
		\]
		and with the above \(R\). Let \(\ell_0>0\) be the corresponding width.
		Choose \(0<\delta\le\delta_0\) so small that
		\[
		\delta<\frac{\ell_0}{4}.
		\]
		Set
		\[
		K=
		\left\{
		y\in\R^N:
		\rho_0\le |y|\le R,\quad
		|y-x|\ge \lambda_*+\frac{3\ell_0}{4}
		\right\}.
		\]
		This is a compact subset of
		\(\Sigma_{\lambda_*}\setminus\partial B_{\lambda_*}(x)\). If \(K\ne
		\varnothing\), strict positivity gives
		\(\min_Kw_{\lambda_*}>0\). If \(K=\varnothing\), no compact-set argument is
		needed. The reflected images of \(K\) remain a positive distance from
		the puncture for all \(\lambda\) in the chosen interval. By uniform
		continuity in \((y,\lambda)\), after decreasing \(\delta\) if necessary, we have
		\[
		w_\lambda(y)\ge0
		\quad\text{for } y\in K,
		\quad \lambda\in[\lambda_*,\lambda_*+\delta].
		\]
		Since \(\delta<\ell_0/4\), the condition \(|y-x|\ge\lambda+\ell_0\) implies \(y\in K\) whenever \(y\in B_R\setminus B_{\rho_0}(0)\) and \(\lambda\in[\lambda_*,\lambda_*+\delta]\).
		Combining this with the estimates near the origin and at infinity, any point
		at which \(w_\lambda<0\) must lie in the narrow region
		\[
		\Omega_\lambda
		=
		\{y:\lambda<|y-x|<\lambda+\ell_0\}\cap B_R .
		\]
		Since \(w_\lambda\ge0\) in \(\Sigma_\lambda\setminus\Omega_\lambda\), the
		narrow region principle, Lemma~\ref{Lem3.4}, gives
		\[
		w_\lambda\ge0
		\quad\text{in }\Omega_\lambda.
		\]
		Therefore
		\[
		w_\lambda\ge0
		\quad\text{in }\Sigma_\lambda
		\]
		for every \(\lambda\in[\lambda_*,\lambda_*+\delta]\). Together with the
		assumption for \(0<\lambda<\lambda_*\), this proves the lemma.
	\end{proof}
	
	\begin{proposition}
		\label{Prop3.9}
		For every \(x\in\R^N\setminus\{0\}\) and every \(0<\lambda<|x|\),
		\[
		u_{x,\lambda}(y)\le u(y)
		\quad \text{for } y\in\Sigma_{x,\lambda}\setminus\{0\}.
		\]
	\end{proposition}
	
	\begin{proof}
		Define
		\[
		\bar\lambda(x)
		=
		\sup\left\{
		\rho\in(0,|x|):
		u_{x,\lambda}\le u
		\text{ in }\Sigma_{x,\lambda}
		\text{ for every }0<\lambda<\rho
		\right\}.
		\]
		By Lemma~\ref{Lem3.6},
		\[
		\bar\lambda(x)>0.
		\]
		If \(\bar\lambda(x)<|x|\), then Lemma~\ref{Lem3.8} extends the comparison
		beyond \(\bar\lambda(x)\), contradicting the definition of
		\(\bar\lambda(x)\). Hence
		\[
		\bar\lambda(x)=|x|.
		\]
		This proves the proposition.
	\end{proof}
	
	\begin{lemma}
		\label{Lem3.10}
		\(u\) is radially symmetric
		about the origin.
	\end{lemma}
	
	\begin{proof}
		Fix a unit vector \(e\). For \(a>0\) and \(0<\varepsilon<a\), take
		\[
		x=ae,
		\qquad
		\lambda=a-\varepsilon.
		\]
		If \(y\cdot e<0\) and \(a\) is large, then \(y\in\Sigma_{x,\lambda}\).
		Proposition~\ref{Prop3.9} gives
		\[
		\left(\frac{a-\varepsilon}{|y-ae|}\right)^{N-2s}
		u\left(ae+\frac{(a-\varepsilon)^2(y-ae)}{|y-ae|^2}\right)
		\le u(y).
		\]
		Letting \(\varepsilon\to0^+\) and then \(a\to+\infty\), we obtain
		\[
		u(y-2(y\cdot e)e)\le u(y)
		\quad \text{for }y\cdot e<0.
		\]
		Replacing \(e\) by \(-e\) gives the reverse inequality. Hence \(u\) is
		symmetric with respect to every hyperplane through the origin.
	\end{proof}
	
	\begin{lemma}
		\label{Lem3.11}
		The radial profile of \(u\)
		is nonincreasing.
	\end{lemma}
	
	\begin{proof}
		By Lemma~\ref{Lem3.10}, \(u(y)=U(|y|)\). Let \(0<r_1<r_2\), choose a unit vector
		\(e\), and take \(a>r_2\). Set
		\[
		x=ae,
		\qquad
		\lambda=\sqrt{(a-r_1)(a-r_2)}.
		\]
		Then \(0<\lambda<a\), \(r_1e\in\Sigma_{x,\lambda}\), and
		\[
		(r_1e)^{x,\lambda}=r_2e.
		\]
		Proposition~\ref{Prop3.9} yields
		\[
		\left(\frac{\lambda}{a-r_1}\right)^{N-2s}U(r_2)\le U(r_1).
		\]
		Letting \(a\to+\infty\), we obtain
		\[
		U(r_2)\le U(r_1).
		\]
		Thus \(U\) is nonincreasing.
	\end{proof}
	
	\begin{lemma}
		\label{Lem3.12}
		The radial profile of \(u\) is strictly decreasing.
	\end{lemma}
	
	\begin{proof}
		We use the representation \eqref{eq3.1} after radial symmetry and
		nonincreasing behavior have been established. We first record a reflection
		identity for radial convolutions. Let \(0<r_1<r_2\), fix a unit vector
		\(e\), set \(t=(r_1+r_2)/2\), and let \(S\) be reflection in the hyperplane
		\(\{y:y\cdot e=t\}\). On the half-space
		\[
		P=\{y:y\cdot e<t\},
		\]
		one has
		\[
		|y|<|Sy|,
		\qquad
		|r_1e-y|<|r_2e-y|.
		\]
		If \(f\ge0\) is radial and nonincreasing and both convolutions below are
		finite, splitting the integrals over \(P\) and \(S(P)\) gives, for
		\(0<\nu<N\),
		\begin{align}
			\label{eq3.25}
			&\int_{\R^N}|r_1e-y|^{-\nu}f(y)\,dy
			-
			\int_{\R^N}|r_2e-y|^{-\nu}f(y)\,dy \notag\\
			&\quad=
			\int_P
			\left(
			|r_1e-y|^{-\nu}-|r_2e-y|^{-\nu}
			\right)
			\left(f(y)-f(Sy)\right)dy
			\ge0.
		\end{align}
		The identity is legitimate because the two original nonnegative integrals
		are finite.
		
		Rotational invariance of the kernel first shows that \(\Hs[u]\) is radial.
		Since \(u^p\) is radial and nonincreasing, \eqref{eq3.25} with
		\(\nu=\mu\) shows that
		\[
		H=\Hs[u]
		\]
		is radial and nonincreasing. Therefore
		\[
		F=Hu^q
		\]
		is positive, radial, and nonincreasing. Applying \eqref{eq3.25} with
		\(\nu=\gamma=N-2s\) gives
		\[
		V(r_1e)\ge V(r_2e).
		\]
		This inequality is strict. If equality held, the first factor in the
		integrand in \eqref{eq3.25} would be strictly positive almost everywhere
		on \(P\), and hence \(F(y)=F(Sy)\) almost everywhere. Radiality already
		makes \(F\) invariant under reflection in \(\{y:y\cdot e=0\}\). Composing
		the two reflections would make \(F\) invariant almost everywhere under
		translation by \(2te\). The difference \(F(y)-F(Sy)\) is continuous away
		from \(0\) and \(S^{-1}(0)\), because
		\(F=(-\Delta)^su\in C_{\loc}(\R^N\setminus\{0\})\). Its almost-everywhere
		vanishing therefore implies pointwise vanishing wherever both terms are
		defined. Hence this translation invariance also holds on the positive ray. A
		radial nonincreasing function with a nonzero translation period
		is constant: its restriction to the ray
		\(\{re:r>0\}\) is both nonincreasing and periodic, and radiality then
		extends the constant value to every sphere. This is impossible here, because \(F>0\)
		and a positive constant violates \eqref{eq:source-condition}. Thus
		\[
		V(r_1e)>V(r_2e).
		\]
		Finally,
		\[
		U(r)=V(r)+mc_{N,s}r^{-\gamma}.
		\]
		The first term is strictly decreasing, and the second is nonincreasing
		when \(m\ge0\). Hence \(U(r_1)>U(r_2)\).
	\end{proof}
	
	\begin{proof}[Proof of Theorem~\ref{Thm1.2}]
		Proposition~\ref{Prop3.9} gives the moving spheres inequality. Lemma~\ref{Lem3.10}
		yields radial symmetry, Lemma~\ref{Lem3.11} gives nonincreasing behavior,
		and Lemma~\ref{Lem3.12} proves strict radial decrease.
	\end{proof}
	
	\section{Homogeneous singular solutions}
	\label{sec:profiles}
	
	This section proves Theorem~\ref{Thm1.3} by an explicit calculation in the
	radial homogeneous class. 
	
	\begin{lemma}
		\label{Lem4.1}
		Let \(0<s<1\), \(N>2s\), and
		\[
		0<\alpha<N-2s.
		\]
		Then
		\[
		(-\Delta)^s |x|^{-\alpha}
		=
		\Lambda_{N,s}(\alpha)|x|^{-\alpha-2s}
		\quad \text{in }\R^N\setminus\{0\},
		\]
		where \(\Lambda_{N,s}(\alpha)\) is defined by
		\eqref{eq:lambda-coefficient}.
	\end{lemma}
	
	\begin{proof}
		With the Fourier convention fixed in Section~\ref{sec:intro}, for
		\(0<\beta<N\),
		\[
		\widehat{|x|^{-\beta}}(\xi)
		=
		d_{N,\beta}|\xi|^{\beta-N},
		\qquad
		d_{N,\beta}
		=
		2^{N-\beta}\pi^{\frac{N}{2}}
		\frac{\Gamma\left(\frac{N-\beta}{2}\right)}{\Gamma\left(\frac{\beta}{2}\right)},
		\]
		and therefore
		\[
		\widehat{(-\Delta)^s |x|^{-\alpha}}(\xi)
		=
		|\xi|^{2s}d_{N,\alpha}|\xi|^{\alpha-N}
		=
		d_{N,\alpha}|\xi|^{\alpha+2s-N}.
		\]
		Since
		\[
		\widehat{|x|^{-\alpha-2s}}(\xi)
		=
		d_{N,\alpha+2s}|\xi|^{\alpha+2s-N},
		\]
		we obtain
		\[
		(-\Delta)^s |x|^{-\alpha}
		=
		\frac{d_{N,\alpha}}{d_{N,\alpha+2s}}|x|^{-\alpha-2s}.
		\]
		The quotient is \(\Lambda_{N,s}(\alpha)\). All gamma-function arguments
		in \eqref{eq:lambda-coefficient} are positive under
		\(0<\alpha<N-2s\), so the coefficient
		is positive. The identity follows from the Fourier transform formula for the Riesz kernel
		recorded in \cite[Chapter~V, Section~1]{Stein}.
	\end{proof}
	
	\begin{lemma}
		\label{Lem4.2}
		Let \(0<\mu<N\) and assume that
		\[
		0<\alpha p<N,
		\qquad
		\mu+\alpha p>N.
		\]
		Then
		\begin{equation}
			\label{eq4.1}
			\int_{\R^N}
			\frac{|y|^{-\alpha p}}{|x-y|^\mu}\,dy
			=
			\mathcal C_{N,\mu,\alpha,p}|x|^{N-\mu-\alpha p}
			\quad \text{for }x\ne0,
		\end{equation}
		where \(\mathcal C_{N,\mu,\alpha,p}\) is defined by
		\eqref{eq:hartree-coefficient}.
	\end{lemma}
	
	\begin{proof}
		The restrictions \(\alpha p<N\), \(\mu<N\), and
		\(\mu+\alpha p>N\) give integrability near \(y=0\), near \(y=x\), and
		at infinity, respectively. The Riesz convolution formula is
		\[
		\int_{\R^N}
		\frac{1}{|x-y|^a|y|^b}\,dy
		=
		C_{N,a,b}|x|^{N-a-b},
		\]
		valid when
		\[
		0<a<N,
		\qquad
		0<b<N,
		\qquad
		a+b>N.
		\]
		It follows by applying the Fourier transform to the two Riesz kernels and
		using their semigroup identity. With the Fourier convention fixed above,
		\[
		C_{N,a,b}
		=
		\frac{d_{N,a}d_{N,b}}{d_{N,a+b-N}},
		\]
		where \(d_{N,\cdot}\) is the coefficient used in
		Lemma~\ref{Lem4.1}. Taking \(a=\mu\) and \(b=\alpha p\) gives
		\eqref{eq4.1} and the coefficient \eqref{eq:hartree-coefficient}; see
		\cite[Chapter~V, Section~1]{Stein} and \cite{LiebLoss}. Every
		gamma-function argument in
		\eqref{eq:hartree-coefficient} is positive under the stated restrictions.
	\end{proof}
	
	\begin{proof}[Proof of Theorem~\ref{Thm1.3}]
		By Lemma~\ref{Lem4.1},
		\[
		A(-\Delta)^s |x|^{-\alpha}
		=
		A\Lambda_{N,s}(\alpha)|x|^{-\alpha-2s}.
		\]
		By Lemma~\ref{Lem4.2},
		\[
		\Hs[u_*](x)
		=
		A^p\mathcal C_{N,\mu,\alpha,p}|x|^{N-\mu-\alpha p}.
		\]
		Therefore
		\[
		\Hs[u_*](x)u_*^q(x)
		=
		A^{p+q}\mathcal C_{N,\mu,\alpha,p}|x|^{N-\mu-\alpha(p+q)}.
		\]
		The equality \((-\Delta)^s u_*=\Hs[u_*]u_*^q\) holds for every
		\(x\ne0\) if and only if
		\[
		-\alpha-2s=N-\mu-\alpha(p+q)
		\]
		and
		\[
		A\Lambda_{N,s}(\alpha)
		=
		A^{p+q}\mathcal C_{N,\mu,\alpha,p}.
		\]
		The exponent identity is precisely \eqref{eq:hom-alpha}, and the coefficient
		identity is \eqref{eq:hom-coefficient}. Since \(p+q-1>0\) and both
		coefficients are positive, \eqref{eq:hom-coefficient} has exactly one
		positive solution \(A\).
		Moreover, \(0<\alpha<N-2s<N\) gives \(u_*\in L^1_{\loc}(\R^N)\), and
		\[
		\int_{\R^N}\frac{u_*(x)}{1+|x|^{N+2s}}\,dx<+\infty .
		\]
		Thus \(u_*\in\Ls(\R^N)\), and Lemma~\ref{Lem4.2} ensures that its Hartree
		potential is finite away from the origin.
		
		For a general homogeneous ansatz \(B|x|^{-\beta}\) under the stated
		convergence conditions, the same calculation first forces
		\[
		-\beta-2s=N-\mu-\beta(p+q),
		\]
		so \(\beta=\alpha\), and then forces \(B=A\).
	\end{proof}
	
	\begin{remark}
		Assume
		\[
		0<s<1,\qquad N>2s,\qquad 0<\mu<N,\qquad p,q>0,\qquad p+q>1.
		\]
		Let \(u\) be a positive punctured solution and let
		\(\alpha\) be given by \eqref{eq:hom-alpha}. Suppose that, for a sequence
		\(r_n\to0^+\),
		\[
		r_n^\alpha u(r_nx)\longrightarrow U(x)
		\quad\text{locally uniformly in }\R^N\setminus\{0\}.
		\]
		Assume additionally that \(U\) is nonzero, positive, radial, and homogeneous
		of degree \(-\alpha\), namely,
		\[
		U(tx)=t^{-\alpha}U(x)
		\quad\text{for every }t>0\text{ and }x\in\R^N\setminus\{0\},
		\]
		and that \(U\) satisfies the limiting Hartree equation and \(\alpha\) satisfies
		\eqref{eq:hom-convergence}. Then Theorem~\ref{Thm1.3} gives
		\[
		U(x)=A|x|^{-\alpha},
		\]
		with \(A\) determined by \eqref{eq:hom-coefficient}. 
	\end{remark}
	
	\begin{remark}
		The restrictions in \eqref{eq:hom-convergence} are exactly the admissibility
		conditions for the preceding positive homogeneous calculation. At \(\alpha=0\), the ansatz
		is constant and its fractional Laplacian vanishes. At \(\alpha=N-2s\), it is
		the fundamental solution and its fractional Laplacian vanishes away from the
		puncture. The endpoint \(\alpha p=N\) gives logarithmic divergence at the
		origin, while \(\alpha p>N\) gives power divergence there. The endpoint
		\(\mu+\alpha p=N\) gives logarithmic divergence at infinity, while
		\(\mu+\alpha p<N\) gives power divergence there. If \(\alpha p\le0\), the
		Hartree integral also diverges at infinity. Finally, the exponent balance
		for a singular radial homogeneous ansatz \(B|x|^{-\beta}\), with
		\(\beta>0\), can be written as
		\[
		\beta(p+q-1)=N-\mu+2s>0.
		\]
		Consequently, if \(p+q\le1\), no positive radial homogeneous solution
		\(B|x|^{-\beta}\), with \(\beta>0\), exists under the stated convergence
		conditions.
	\end{remark}
	
	\section*{Acknowledgments}
	This work is supported by National Natural Science Foundation of China (12301145, 12561020, 12261107) and Yunnan Fundamental Research Projects (202401AU070123, 202601AT070048).
	
	\medskip
	{\bf Author Contributions:} All authors contributed equally to the writing and preparation of the manuscript.
	
	\medskip
	{\bf Data availability:} Data sharing is not applicable to this article as no new data were created or analyzed in this study.
	
	\medskip
	{\bf Conflict of interest:} The authors declare that they have no conflict of interest.

\end{document}